\documentclass[11pt]{article}
\usepackage{amsthm}
\usepackage{amsmath}
\usepackage{amssymb}
\usepackage{graphicx}
\usepackage{float}
\usepackage{subfigure}

\usepackage[margin=3cm]{geometry}

\theoremstyle{plain}
\newtheorem{theorem}{Theorem}[section]

\newtheorem{lemma}[theorem]{Lemma}
\newtheorem{proposition}[theorem]{Proposition}

\newtheorem{definition}[theorem]{Definition}

\usepackage[ruled,linesnumbered] {algorithm2e}

\usepackage{amsmath}
\usepackage{amsthm}
\usepackage{amssymb}

\newcommand{\reals}{\mathbb{R}}
\newcommand{\Rn}{\mathbb{R}^n}
\newcommand{\Rm}{\mathbb{R}^m}
\newcommand{\calM}{\mathcal{M}}
\newcommand{\Exp}{\operatorname{Exp}}
\newcommand{\grad}{\operatorname{grad}}
\newcommand{\Hess}{\operatorname{Hess}}
\newcommand{\T}{\mathrm{T}}

\usepackage{color}

\usepackage[colorlinks=true,bookmarks=false,linkcolor=blue,urlcolor=blue,citecolor=blue,breaklinks=true]{hyperref}

\begin{document}


\title{Simple algorithms for optimization on Riemannian manifolds\\with constraints}

\author{
	Changshuo Liu
\and Nicolas Boumal\thanks{The authors are with Princeton University, PACM and Mathematics Department. \newline \indent \ \ \ Contact: \texttt{nboumal@math.princeton.edu}.}
}

\maketitle

\begin{abstract}
We consider optimization problems on manifolds with equality and inequality constraints. A large body of work treats constrained optimization in Euclidean spaces. In this work, we consider extensions of existing algorithms from the Euclidean case to the Riemannian case. Thus, the variable lives on a known smooth manifold and is further constrained. In doing so, we exploit the growing literature on unconstrained Riemannian optimization. For the special case where the manifold is itself described by equality constraints, one could in principle treat the whole problem as a constrained problem in a Euclidean space. The main hypothesis we test here is whether it is sometimes better to exploit the geometry of the constraints, even if only for a subset of them. Specifically, this paper extends an augmented Lagrangian method and smoothed versions of an exact penalty method to the Riemannian case, together with some fundamental convergence results. Numerical experiments indicate some gains in computational efficiency and accuracy in some regimes for minimum balanced cut, non-negative PCA and $k$-means, especially in high dimensions.



\end{abstract}

\section{Introduction}
We consider the following problem:
    \begin{equation}\label{problem:mcp}
    \begin{aligned}
    & \underset{x}{\text{min}}
    & & f(x)\\
    & \text{subject to}
    & & x \in {\cal{M}} \\
    & & & g_i(x)\leq 0 \text{ for } i \in {\cal{I}} = \{1, \dots, n\}, \\
   & & & h_j(x) = 0 \text{ for } j \in {\cal{E}}  = \{n+1, \dots, n+m\},
    \end{aligned}
    \end{equation}
where ${\cal{M}}$ is a Riemannian manifold and $f, \{g_i\}, \{h_j\}$ are twice continuously differentiable functions from ${\cal{M}}$ to $\mathbb{R}$. The problems of this class have extra constraints in addition to the manifold constraint. Following the convention, we call problem~\eqref{problem:mcp} an Equality Constrained Problem (ECP) when only equality constraints exist, and an Inequality Constrained Problem (ICP) when only inequality constraints exist. If both equality and inequality constraints are present, we call it a Mixed Constrained Problem (MCP). Such problems feature naturally in applications. For instance, non-negative principal component analysis (PCA) is formulated as an optimization problem on a sphere in $\Rn$ with non-negativity constraints on each entry~\cite{zass2007nonnegative}. As another example, $k$-means can be formulated as a constrained optimization problem on the Stiefel manifold~\cite{carson2017manifold}. We discuss these more in Section~\ref{sec:XP}.

Necessary and sufficient optimality conditions for the general problem class (\ref{problem:mcp}) were derived in~\cite{yang2014optimality} and also recently in~\cite{bergmann2018intrinsic}---we summarize them in the next section. Some algorithmic approaches have been put forward in \cite{dreisigmeyer2006equality, kovnatsky2016madmm, birgin2016augmented,khuzani2017stochastic, zhang2017primal}. Nevertheless, and somewhat surprisingly, we find that there has been no systematic effort to survey and compare some of the most direct approaches to solve~\eqref{problem:mcp} based on prior work on the same problem class without the manifold constraint~\cite{bertsekas2014constrained} and with only the manifold constraint~\cite{absil2009optimization}.

Part of the reason may be that, in many applications, the manifold ${\cal{M}}$ is a submanifold of a Euclidean space, itself defined by equality constraints. In such cases, the manifold constraint can be treated as an additional set of equality constraints, and the problem can be solved using the rich expertise gained over the years for constrained optimization in $\mathbb{R}^n$. There are also existing software packages for it, such as Lancelot, KNITRO and Algencan \cite{birgin2014practical, conn2013lancelot, byrd2006knitro}.

Yet, based on the literature for unconstrained optimization on manifolds, we see that if the manifold ${\cal{M}}$ is nice enough, it pays to exploit its structure fully. In particular, much is now understood about optimizing over spheres, orthogonal groups, the Stiefel manifold of orthonormal matrices, the set of fixed-rank matrices, and many more. Furthermore, embracing the paradigm of optimization on manifolds also allows us to treat problems of class (\ref{problem:mcp}) where ${\cal{M}}$ is an abstract manifold, such as the Grassmannian manifold of linear subspaces. Admittedly, owing to Whitney's embedding theorem, abstract manifolds can also be embedded in a Euclidean space, and hence even those problems could in principle be treated using algorithms from the classical literature, but the mere existence of an embedding is often of little practical use.

In this paper, we survey some of the more classical methods for constrained optimization in $\mathbb{R}^n$ and straightforwardly extend them to the more general class (\ref{problem:mcp}), while preserving and exploiting the smooth geometry of $\cal{M}$. For each method, we check if some of the essential convergence results known in $\mathbb{R}^n$ extend as well. Then, we set up a number of numerical experiments for a few applications and report performance profiles. Our purpose in doing so is to gain some perspective as to which methods are more likely to yield reliable generic software for this problem class.

\subsection{Contributions}

As a first contribution, we study the augmented Lagrangian method (ALM) and the exact penalty method on general Riemannian manifolds. For ALM, we study local and global convergence properties. For the exact penalty method, each iteration involves the minimization of a sum of maximum functions, which we call a mini-sum-max problem. For this nonsmooth optimization subproblem, we study and tailor two types of existing algorithms---subgradient descent and smoothing methods. In particular, we propose a robust subgradient method without gradient sampling for mini-sum-max problems, which may be of interest in itself. For smoothing methods, we study the effect of two classical smoothing functions: Log-sum-exp and `Linear-Quadratic + Huber loss'. 

As a second contribution, we perform numerical experiments on non-negative PCA, $k$-means and the minimum balanced cut problems to showcase the strengths and weaknesses of each algorithm. As a baseline, we compare our approach to the traditional approach, which would simply consider the manifold as an extra set of equality constraints. For this, we use \texttt{fmincon}: a general-purpose nonlinear programming solver in Matlab. For these problems, we find that some of our methods perform better than \texttt{fmincon} in high dimensional scenarios. Our solvers are generally slower in the low-dimensional case.

We present the algorithms and convergence analyses for ALM and the exact penalty method on Riemannian manifolds in Sections~\ref{sec:ALM} and~\ref{sec:EPM} respectively. Proofs are deferred to the appendix. Numerical experiments on various applications are reported in Section~\ref{sec:XP}.
 
\subsection{Related literature}
 Dreisigmeyer~\cite{dreisigmeyer2006equality} tackles ICPs on Euclidean submanifolds by pulling back the inequality constraints onto the tangent spaces of $\calM$, and running direct search methods there. Yang et al.~\cite{yang2014optimality} provide necessary optimality conditions for our problem class~\eqref{problem:mcp}. Bergmann and Herzog~\cite{bergmann2018intrinsic} extend a range of constraint qualifications from the Euclidean setting to the smooth manifold setting. Kovnatsky et al.~\cite{kovnatsky2016madmm} generalize the alternating direction method of multipliers (ADMM) to the Riemannian case---the method handles nonsmooth optimization on manifolds via variable splitting, which produces an ECP. In~\cite{birgin2016augmented}, Birgin et al.\ deal with MCP in the Euclidean case by splitting constraints into upper-level and lower-level, and perform ALM on upper-level constrained optimization problems in a search space confined by lower-level constraints. When ${\cal{M}}$ is a submanifold of a Euclidean space chosen to describe the lower-level constraints, our Riemannian ALM reduces to their method. Khuzani and Li~\cite{khuzani2017stochastic} propose a primal-dual method to address ICP on Riemannian manifolds with bounded sectional curvature in a stochastic setting. Zhang et al.~\cite{zhang2017primal} propose an ADMM-like primal-dual method for nonconvex, nonsmooth optimization problems on submanifolds of Euclidean space coupled with linear constraints. Weber and Sra~\cite{weber2017frank} recently study a Franke--Wolfe method on manifolds to design projection-free methods for constrained, geodesically convex optimization on manifolds.

\section{Preliminaries and notations} \label{sec:prelims}

We briefly review some relevant concepts from Riemannian geometry, following the notations of~\cite{absil2009optimization}. Let the Riemannian manifold ${\cal{M}}$ be endowed with a Riemannian metric $\langle\cdot ,\cdot \rangle_x$ on each tangent space $\mathrm{T}_x{\cal{M}}$, where $x$ is in ${\cal{M}}$. Let $\|\cdot\|_x$ be the associated norm. We often omit the subscript $x$ when it is clear from context. Throughout the paper, we assume that ${\cal{M}}$ is a complete, smooth, finite-dimensional Riemannian manifold.

\subsection{Gradients and Hessians on manifolds}

The gradient at $x$ of a smooth function $f \colon {\cal{M}} \to \reals$, $\operatorname{grad} f(x)$, is defined as the unique tangent vector at $x$ such that
\[
    \langle\operatorname{grad} f(x), v \rangle_x = \operatorname{D}\!f(x)[v],\quad \forall v\in \mathrm{T}_x{\cal{M}},
\]
where the right-hand side is the directional derivative of $f$ at $x$ along $v$.
Let $\mathrm{dist}(x,y)$ denote the Riemannian distance between $x,y\in{\cal{M}}$. For each $x\in{\cal{M}}$, let $\mathrm{Exp}_x \colon \mathrm{T}_x{\cal{M}}\rightarrow {\cal{M}}$ denote the exponential map at $x$ (that is, the map such that $t\mapsto\mathrm{Exp}_x(tv)$ is a geodesic passing through $x$ at $t = 0$ with velocity $v \in \mathrm{T}_x\calM$). The injectivity radius is defined as:
\begin{equation}\label{eq:injectivity}
    {\mathit{i}}({\cal{M}}) = \inf_{x\in{\cal{M}}} \sup \{\epsilon>0: \mathrm{Exp}_x|_{\{\eta\in \mathrm{T}_x{\cal{M}}: \|\eta\|< \epsilon\}}\textrm{ is a diffeomorphism}\}.
\end{equation}
For any $x,y\in{\cal{M}}$ with $\mathrm{dist}(x,y) < {\mathit{i}}({\cal{M}})$, there is a unique minimizing geodesic connecting them, which gives rise to a parallel transport operator ${\cal{P}}_{x\rightarrow y}:\mathrm{T}_x{\cal{M}}\rightarrow \mathrm{T}_y{\cal{M}}$ as an isometry between two tangent spaces. 

\begin{definition}[Hessian, \cite{absil2009optimization}, Definition 5.5.1]
    Given a smooth function $f \colon \calM \to \reals$, the $\mathrm{Riemannian\ Hessian}$ of $f$ at a point $x$ in ${\cal{M}}$ is the linear mapping $\mathrm{Hess} f(x)$ of $\mathrm{T}_x{\cal{M}}$ into itself defined by
    \[
        \mathrm{Hess} f(x)[\xi_x] = \nabla_{\xi_x}\operatorname{grad} f
    \]
    for all $\xi_x$ in $\mathrm{T}_x{\cal{M}}$, where $\nabla$ is the Riemannian connection on ${\cal{M}}$.
\end{definition}

Using the exponential map, one can also understand the Riemannian gradient and Hessian through the following identities~\cite[Eq.~(4.4) and Prop.~5.5.4]{absil2009optimization}:
\begin{align}
	\grad f(x) & = \grad(f \circ \Exp_x)(0_x), & \Hess f(x) & = \Hess(f \circ \Exp_x)(0_x),
	\label{eq:gradHesspullbackExp}
\end{align}
where $0_x$ is the zero vector in $\T_x\calM$ and $f \circ \Exp_x \colon \T_x\calM \to \reals$ is defined on a Euclidean space with a metric, so that its gradient and Hessian are defined in the usual sense. The composition $f \circ \Exp_x$ is also called the pullback of $f$ to the tangent space at $x$.

\subsection{Optimality conditions}\label{sec:optcon}
Let $\Omega$ denote the set of feasible points on ${\cal{M}}$ satisfying the constraints in~\eqref{problem:mcp}. At $x\in{\cal{M}}$, let ${\cal{I}}$ denote the set of $n$ inequality constraints, and ${\cal{E}}$ the set of $m$ equality constraints. Let ${\cal{A}}(x)$ denote the active set of constraints, that is,
\begin{align}
	{\cal{A}}(x) & = {\cal{E}}\cup \{i\in {\cal{I}}| g_i(x)=0\}.
\end{align}
The Langrangian of $(\ref{problem:mcp}$) is defined similarly to the Euclidean case as
\begin{equation}\label{lagrangian}
    {\cal{L}}(x,\lambda,\gamma) = f(x) + \sum_{i\in {\cal{I}}} \lambda_i g_i(x) + \sum_{j\in {\cal{E}}} \gamma_j h_j(x),
\end{equation}
where $\lambda_i$, $\gamma_j$ are vector entries of $\lambda \in \mathbb{R}^n$ and $\gamma \in \mathbb{R}^m$. Following \cite{yang2014optimality}, constraint qualifications and optimality conditions are generalized as follows:
\begin{definition}[LICQ, \cite{yang2014optimality}, eq.~(4.3)]
Linear independence constraint qualifications (LICQ) are said to hold at $x\in {\cal{M}}$ if
\begin{equation}\label{def:licq}
    \{\operatorname{grad} g_i(x), \operatorname{grad} h_j(x), i \in {\cal{A}}(x)\cap {\cal{I}}, j\in{\cal{E}} \} \text{ are linearly independent in $\mathrm{T}_x{\cal{M}}$}.
\end{equation}
\end{definition}

\begin{definition}[First-Order Necessary Conditions (KKT conditions), \cite{yang2014optimality}, eq.~(4.8)]\label{prop:fonc} Given an MCP as in (\ref{problem:mcp}), $x^*\in\Omega$ is said to satisfy KKT conditions if there exist Lagrange multipliers $\lambda^*$ and $\gamma^*$ such that the following hold:
    \begin{align}\nonumber
        &\operatorname{grad} f(x^*) + \sum_{i\in {\cal{I}}} \lambda_i^* \operatorname{grad} g_i(x) + \sum_{j\in {\cal{E}}} \gamma_i^* \operatorname{grad} h_j(x) = 0,\\ \label{cond:first}
        &h_j(x^*) = 0, \text{ for all $j\in {\cal{E}}$, and}\\\nonumber
        &g_i(x^*) \leq 0,\ \lambda_i^*\geq 0,\ \lambda_i^*g_i^{}(x^*) = 0  \text{ for all $i\in {\cal{I}}$.}
    \end{align}
\end{definition}

For the purpose of identifying second-order optimality conditions at $x^*$ with associated $\lambda^*$ and $\gamma^*$, consider the critical cone $F(x^*, \lambda^*, \gamma^*)$ inside the tangent space at $x^*$ defined as follows; see \cite[\S4.2]{yang2014optimality}\cite[\S12.4]{Nocedal2006NO}:

\begin{equation}\label{set:f2}
  v \in F(x^*,\lambda^*, \gamma^*) \Leftrightarrow \begin{cases}
            v\in \mathrm{T}_{x^*}{\cal{M}}, &\\
    \langle \operatorname{grad} h_j(x^*), v\rangle = 0& \textrm{ for all }j\in {\cal{E}} ,\\
    \langle \operatorname{grad} g_i(x^*), v\rangle = 0& \textrm{ for all }i\in {\cal{A}}(x^*)\cap{\cal{I}} \textrm{ with }\lambda_i^* >0\textrm{, and}\\
    \langle \operatorname{grad} g_i(x^*), v\rangle \leq 0& \textrm{ for all }i\in {\cal{A}}(x^*)\cap{\cal{I}} \textrm{ with }\lambda_i^* = 0. \\
  \end{cases}
\end{equation}

\begin{definition}[Second-Order Necessary Conditions (SONC), \cite{yang2014optimality}, Theorem 4.2]\label{def:sonc}
    Given an MCP as in (\ref{problem:mcp}), $x^*\in\Omega$ is said to satisfy SONCs if it satisfies KKT conditions with associated Lagrange multiplier $\lambda^*$ and $\gamma^*$, and if 
    \[
        \langle v, \mathrm{Hess}_x {\cal{L}} (x^*,\lambda^*,\gamma^*) v\rangle \geq 0, \textrm{ for any } v\in F(x^*, \lambda^*, \gamma^*),
    \]
    where the Hessian is taken with respect to the first variable of $\mathcal{L}$, on $\calM$.
\end{definition}

\begin{definition}[Second-Order Sufficient Conditions (SOSC), \cite{yang2014optimality}, Theorem 4.3] \label{def:sosc}
    Given an MCP as in (\ref{problem:mcp}), $x^*\in\Omega$ is said to satisfy SOSCs if it satisfies KKT conditions with associated Lagrange multiplier $\lambda^*$ and $\gamma^*$, and if
    \[
        \langle v, \mathrm{Hess}_x {\cal{L}} (x^*,\lambda^*,\gamma^*) v\rangle > 0, \textrm{ for any } v\in F(x^*, \lambda^*, \gamma^*),\ v\neq 0.
    \]
\end{definition}

\begin{proposition}[\cite{yang2014optimality}, Theorem 4.1, 4.2]\label{prop:sonc}
 If $x^*$ is a local minimum of a given MCP and LICQ holds at $x^*$, then $x^*$ satisfies KKT conditions and SONCs.
\end{proposition}
\begin{proposition}[\cite{yang2014optimality}, Theorem 4.3]\label{prop:sosc}
If $x^*$ satisfies SOSCs for a given MCP, then it is a strict local minimum.
\end{proposition}

\section{Riemannian augmented Lagrangian methods} \label{sec:ALM}

The augmented Lagrangian method (ALM) is a popular algorithm for constrained nonlinear programming of the form of (\ref{problem:mcp}) with ${\cal{M}} = \mathbb{R}^{d}$. At its core, ALM relies on the definition of the \emph{augmented} Lagrangian function~\cite[eq.~(4.3)]{birgin2014practical}:
\begin{equation}\label{eq:alm}
    {\cal{L}}_\rho(x, \lambda, \gamma) = f(x) + \frac{\rho}{2}\left(\sum_{j\in {\cal{E}}} \left(h_j(x)+\frac{\gamma_j}{\rho}\right)^2+ \sum_{i\in {\cal{I}}} \max\left\{0, \frac{\lambda_i}{\rho}+g_i(x)\right\}^2 \right),
\end{equation}
where $\rho > 0$ is a penalty parameter, and $\gamma \in \Rm, \lambda \in \Rn, \lambda \geq 0$. ALM alternates between updating $x$ and updating $(\lambda,\gamma)$. To update $x$, any algorithm for unconstrained optimization may be adopted to minimize~\eqref{eq:alm} with $(\lambda,\gamma)$ fixed. We shall call the chosen solver the \emph{subsolver}. To update $(\lambda,\gamma)$, a clipped gradient-type update rule is used.  A vast literature covers ALM in the Euclidean case. We direct the reader in particular to the recent monograph by Birgin and Mart\'inez~\cite{birgin2014practical}.

The Lagrangian function as defined in~\eqref{eq:alm} generalizes seamlessly to the Riemannian case simply by restricting $x$ to live on the manifold ${\cal{M}}$. Importantly, $\mathcal{L}_\rho$ is continuously differentiable in $x$ under our assumptions. The corresponding ALM algorithm is easily extended to the Riemannian case as well: subsolvers are now optimization algorithms for unconstrained optimization on manifolds. We refer to this approach as Riemannian ALM, or RALM; see Algorithm~\ref{algo:alm}, in which the clip operator is defined by
\begin{align*}
	\operatorname{clip}_{[a, b]}(x) = \max\{a, \min(b, x)\}.
\end{align*}

\begin{algorithm}[t]\label{algo:alm}
\SetAlgoLined
\SetKw{KwRequire}{Require: }
\SetKw{KwInput}{Input: }
\KwRequire{Riemannian manifold $\cal{M}$, twice continuously differentiable functions $f$, $\{g_i\}_{i\in{\cal{I}}}$, $\{h_j\}_{j\in{\cal{E}}} \colon \cal{M} \rightarrow \mathbb{R}$.}	 \\
\KwInput{Starting point $x_0\in{\cal{M}}$, starting Lagrangian vectors $\lambda^0 \in \mathbb{R}^n$, $\gamma^0 \in \mathbb{R}^m$, accuracy tolerance $\epsilon_{\min}$, starting accuracy $\epsilon_0 > 0$, starting penalty coefficient $\rho_0$, constants $\theta_\epsilon \in (0,1)$, $\theta_\rho > 1$, multiplier boundaries $\lambda^{\max}\in \mathbb{R}^n$, $\gamma^{\min}, \gamma^{\max}\in \mathbb{R}^m$ with $\gamma^{\min}_i \leq \gamma^{\max}_i$ for each $i\in{\cal{I}}$, ratio $\theta_{\sigma}\in (0,1)$, minimum step size $d_{\min}$.} \\
 \For{$k = 0,1,\dots$}{
    
    Compute $x_{k+1}$---an approximate solution to the following problem within a tolerance $\epsilon_{k}$:
    \begin{equation}\label{eq:almprob}
    \begin{aligned}
    & \underset{x\in {\cal{M}}}{\text{min}}
    & & {\cal{L}}_{\rho_k}(x, \lambda^k, \gamma^k).\\
    \end{aligned}
    \end{equation}

    \If{$\mathrm{dist}(x_k, x_{k+1}) < d_{\min}$ $\mathrm{and}$ $\epsilon_k \leq \epsilon_{\min}$}{
        Return $x_{k+1}$\;
    }
    $\gamma_j^{k+1} = \operatorname{clip}_{[\gamma_j^{\min}, \gamma_j^{\max}]}(\gamma_j^k + \rho_k h_j(x_{k+1}))$, for $j\in {\cal{E}}$\;
    $\lambda_i^{k+1}= \operatorname{clip}_{[\lambda_i^{\min},\lambda_i^{max}]}(\lambda_i^k + \rho_k g_i(x_{k+1}))$, for $i\in {\cal{I}}$\;
    $\sigma^{k+1}_i = \max\left\{g_i(x_{k+1}), - \frac{\lambda^{k}_i}{\rho_k}\right\}$, for $i\in {\cal{I}}$\;
    $\epsilon_{k+1} = \max\left\{\epsilon_{\min}, \theta_\epsilon \epsilon_k\right\}$\;
    \eIf{$k=0$ $\mathrm{or}$ $\max_{j\in {\cal{E}}, i\in {\cal{I}}}\left\{|h_j(x_{k+1})|, |\sigma^{k+1}_i|\right\} \leq \theta_{\sigma}\max_{j\in {\cal{E}}, i\in {\cal{I}}}\left\{|h_j(x_{k})|, |\sigma^{k}_i|\right\}$}{
        $\rho_{k+1} = \rho_k$\;
    }{
    	$\rho_{k+1} = \theta_\rho \rho_k$\;
	}
 }
\caption{Riemannian augmented Lagrangian method (RALM)}
\end{algorithm}
In practice, solving~\eqref{eq:almprob} (approximately) involves running any standard algorithm for smooth, unconstrained optimization on manifolds with warm-start at $x_k$. Various kinds of tolerances for this subproblem solve will be discussed, which lead to different convergence results. We note that this approach of separating out a subset of the constraints that have special, exploitable structure in ALM is in the same spirit as the general approach in~\cite{andreani2007augmented} and in~\cite[\S2.4]{bertsekas2014constrained}.

 Notice that, in Algorithm \ref{algo:alm}, there are safeguards designed for multipliers and a conditioned update on the penalty coefficient $\rho$---updates are executed only when constraint violations are shrinking fast enough, which helps alleviate the effect of ill-conditioning and improves robustness; see~\cite{birgin2014practical, kanzow2017example}. As each subproblem (\ref{eq:almprob}) is an unconstrained problem with sufficiently smooth objective function, various Riemannian optimization methods can be used. In particular, we mention the Riemannian gradient descent, non-linear conjugate gradients and trust-regions methods, all described in~\cite{absil2009optimization} and available in ready-to-use toolboxes~\cite{boumal2014manopt, pymanopt, roptlib}.


Global and local convergence properties of ALM in the Euclidean case have been studied extensively; see~\cite{andreani2007augmented, andreani2011sequential, bertsekas1999nonlinear, birgin2016augmented, birgin2010global}, among others. Some of the results available in the literature are phrased in sufficiently general terms that they readily apply to the Riemannian case, even though the Riemannian case was not necessarily explicitly acknowledged. This applies to Proposition~\ref{prop:almglobal} below. To extend certain other results, minor modifications to the Euclidean proofs are necessary. Propositions~\ref{prop:almfirst} and \ref{prop:almsecond} are among those. Here and in the next section, we state some of the more relevant results explicitly for the Riemannian case. For the sake of completeness, we include full proofs in the appendix, stressing again that they are either existing proofs or simple adaptations of existing proofs.

Whether or not Algorithm~\ref{algo:alm} converges depends on the tolerances for the subproblems, and the ability of the subsolver to return a point that satisfies them. In general, solving the unconstrained subproblem~\eqref{eq:alm} to global optimality---even within a tolerance on the objective function value---is hard. In practice, that step is implemented by running a local optimization solver which, often, can only guarantee convergence to an approximate first- or second-order stationary point. Nevertheless, practice suggests that for many applications these local solvers perform well. An important question then becomes: assuming the subproblems are indeed solved within appropriate tolerances, does Algorithm~\ref{algo:alm} converge? In the following three propositions, we partially characterize the limit points generated by the algorithm assuming either that the subproblems are solved almost to global optimality, which is difficult to guarantee, or assuming approximate stationary points are computed, which we can guarantee~\cite{boumal2016globalrates,bento2017iterationcomplexity,zhang2018cubicregmanifold,agarwal2018arcfirst}.

We first consider the case when we have a global subsolver up to some tolerance on the cost for each iteration. This affords the following result:
\begin{proposition}\label{prop:almglobal}
In Algorithm \ref{algo:alm} with $\epsilon_{\min} = 0$ (so that the algorithm produces an infinite sequence with $\epsilon_k \to 0$), if at each iteration $k$ the subsolver produces a point $x_{k+1}$ satisfying
\begin{equation}\label{eq:almstoppingcrt2}
    {\cal{L}}_{\rho_k}(x_{k+1},\lambda^k,\gamma^k) \leq {\cal{L}}_{\rho_k}(z,\lambda^k,\gamma^k) + \epsilon_k,
\end{equation}
where $z$ is a feasible global minimizer of the original MCP, and if $\{x_k\}_{k=0}^\infty$ has a limit point $\overline{x}$, then $\overline{x}$ is a global minimizer of the original MCP.
\end{proposition}
\begin{proof}
See Theorems 1 and 2 in \cite{birgin2010global}. The proofs in that reference apply verbatim to the Riemannian case.
\end{proof}

However, most Riemannian optimization methods only return approximately first- or second-order approximate stationary points. In the first-order case, we have the following:

\begin{proposition}\label{prop:almfirst}
In Algorithm \ref{algo:alm} with $\epsilon_{\min} = 0$, if at each iteration $k$ the subsolver produces a point $x_{k+1}$ satisfying
\begin{equation}\label{eq:almstoppingcrt1}
    \|\operatorname{grad}_x {\cal{L}}_{\rho_k}(x_{k+1}, \lambda^k, \gamma^k)\| \leq \epsilon_k,
\end{equation}
and if the sequence $\{x_k\}_{k=0}^\infty$ has a limit point $\overline{x}\in\Omega$ where LICQ is satisfied, then $\overline{x}$ satisfies KKT conditions of the original MCP.
\end{proposition}
\begin{proof}
The proof is an easy adaptation of that of Theorem 4.2 in~\cite{andreani2007augmented}: see Appendix~\ref{appendixalmfirst}.
\end{proof}

In the second-order case, we consider problems with equality constraints only, because when inequality constraints are present the augmented function (\ref{eq:alm}) may not be twice differentiable. (See \cite{bertsekas2014constrained} for possible solutions to this particular issue.) We consider the notion of Weak Second-Order Necessary Conditions on manifolds, which parallels the Euclidean case definition in~\cite{andreani2017second, guo2013second}. 
\begin{definition}[Weak Second-Order Necessary Conditions (WSONC)] \label{def:wsonc}
    Given an MCP as in \eqref{problem:mcp}, a feasible point $x^*\in \Omega$ satisfies WSONC if it satisfies KKT conditions with multipliers $\lambda^*\in \mathbb{R}^n_+$ and $\gamma^*\in \mathbb{R}^m$ such that\footnote{Note that this condition involves $\cal{L}$ as defined in Section~\ref{sec:optcon}, not ${\cal{L}}_\rho$.}
    \[
        \langle v, \mathrm{Hess}_{x}{\cal{L}}(x^*,\lambda^*)v\rangle \geq 0,\textrm{ for any } v \in {\cal{C}}^W(x^*),
    \]
    where 
    \[
        v \in {\cal{C}}^W(x^*) \Leftrightarrow \begin{cases}
            v\in \mathrm{T}_{x^*}{\cal{M}}, &\\
        \langle \operatorname{grad} h_j(x^*), v\rangle = 0& \textrm{ for all }j\in {\cal{E}},\textrm{ and}\\
        \langle \operatorname{grad} g_i(x^*), v\rangle = 0& \textrm{ for all }i\in {\cal{A}}(x^*)\cap{\cal{I}}.\\
        \end{cases}
    \]
\end{definition}
Note that the cone ${\cal{C}}^W$ in Definition~\ref{def:wsonc} is smaller than the cone $F$~\eqref{set:f2} used to define SONC. Even in Euclidean space, most algorithms do not converge to feasible SONC points in general, so we would not expect more in the Riemannian case; see~\cite{andreani2017second, gould1999note}. $C^W$ is called weak critical cone in~\cite[Def.~2.2]{andreani2017second}.

\begin{proposition}\label{prop:almsecond}
Consider an ECP (problem~\eqref{problem:mcp} without inequalities) with $|\cal{E}| < \mathrm{dim}({\cal{M}})$. In Algorithm~\ref{algo:alm} with $\epsilon_{\min} = 0$, suppose that at each iteration $k$ the subsolver produces a point $x_{k+1}$ satisfying
\begin{equation}\label{eq:almstoppingcrt11}
    \|\operatorname{grad}_x {\cal{L}}_{\rho_k}(x_{k+1}, \lambda^k, \gamma^k)\| <\epsilon_k
    \end{equation} 
    and 
    \begin{equation}\label{eq:almstoppingcrt12}
    \mathrm{Hess}_x {\cal{L}}_{\rho_k}(x_{k+1}, \lambda^k, \gamma^k) \succeq -\epsilon_k I
    \end{equation}
    (meaning all eigenvalues of the Hessian are at or above $-\epsilon_k$).
    If the sequence generated by the algorithm has a limit point $\overline{x}\in\Omega$ where LICQ is satisfied, then $\overline{x}$ satisfies the WSONC conditions of the original ECP.
\end{proposition}
\begin{proof}
	The proof is adapted from Section 3 in~\cite{andreani2017second}: see Appendix~\ref{appendixalmsecond}.
\end{proof}

\section{Exact penalty method} \label{sec:EPM}
Another standard approach to handle constraints is the so-called exact penalty method. As a replacement for the constraints, the method supplements the cost function with a weighted $L_1$ penalty for violating the constraints. This leads to a nonsmooth, unconstrained optimization problem.
In the Riemannian case, to solve problem (\ref{problem:mcp}), this approach suggests solving the following program, where $\rho > 0$ is a penalty weight:
    \begin{equation}\label{l1sumpenalty:eq}
    \begin{aligned}
    & \underset{x\in {\cal{M}}}{\text{min}}
    & & f(x) + \rho \left(\sum_{i\in{\cal{I}}}\text{max}\{0, g_i(x)\} + \sum_{j\in{\cal{E}}}|h_j(x)|\right).\\ 
    \end{aligned}
    \end{equation} 
In the Euclidean case, it is known that only a finite penalty weight $\rho$ is needed for exact satisfaction of constraints, hence the name~\cite[Ch.~15, 17]{Nocedal2006NO}. We have the following analogous property in the Riemannian case.
\begin{proposition}\label{l1sum:prop}
    If $x^*$ is a local minimum for problem~\eqref{problem:mcp} and it satisfies SOSC with KKT multipliers $\lambda^*, \gamma^*$, then for any $\rho$ such that $\rho > \max_{i\in{\cal{I}}, j\in{\cal{E}}}\{|\lambda_i^*|, |\gamma_j^*|\}$, $x^*$ is also a local minimum for~\eqref{l1sumpenalty:eq}.
\end{proposition}
\begin{proof}
The proof resembles that of Theorem 6.9 in \cite{ruszczynski2006nonlinear}: see Appendix~\ref{appendl1sum:prop}.
\end{proof}
Notice that the resulting penalized cost function is nonsmooth: such problems may be challenging in general. Fortunately, the cost function in~\eqref{l1sumpenalty:eq} is a sum of maximum functions and absolute values functions: this special structure can be exploited algorithmically. We explore two general strategies: first, we explore smoothing techniques; then, we explore a particular nonsmooth Riemannian optimization algorithm.


\subsection{Smoothing technique}\label{sec:REPMS}

We discuss two smoothing methods for~\eqref{l1sumpenalty:eq}. A first approach is to note that the absolute value function can be written as a max function: $|x| = \max\{x, -x\}$. Thus, the nonsmooth part of the cost function in~\eqref{l1sumpenalty:eq} is a linear combination of max functions between two terms. A popular smoothing for a two-term maximum is the log-sum-exp function~\cite{chen1995smoothing}: $\max\{a, b\} \approx u\log(e^{a/u}+e^{b/u})$, with smoothing parameter $u > 0$. This yields the following smooth, unconstrained problem on a manifold, where $\rho, u$ are fixed constants:
    \begin{equation}\label{eq:qlse}
    \begin{aligned}
    & \underset{x\in {\cal{M}}}{\text{min}}
    & & Q^{\textrm{lse}}(x,\rho, u) = f(x) +\rho\sum_{i\in{\cal{I}}} u\log (1+e^{g_i(x)/u}) +\rho\sum_{j\in{\cal{E}}} u\log(e^{h_j(x)/u}+e^{-h_j(x)/u}).\\ 
    \end{aligned}
    \end{equation}
Another common approach is to smooth the absolute value and the max in~\eqref{l1sumpenalty:eq} separately, using respectively a pseudo-Huber loss~\cite{cambier2016robust} and a linear-quadratic loss~\cite{pinar1994smoothing}. Here, still with smoothing parameter $u > 0$, we use $|x| \approx \sqrt{x^2 + u^2}$ and $\max\{0, x\} \approx {\cal{P}}(x,u)$ where
\[
{\cal{P}}(x,u) =  \begin{cases}
            0 &\textrm{ if } x\leq 0\\
            \frac{x^2}{2u}&\textrm{ if } 0\leq x\leq u\\
            x-\frac{u}{2} &\textrm{ if } x\geq u.\\
        \end{cases}
\] Such approximation yields, for fixed constants $\rho, u$, the following problem:
    \begin{equation}\label{eq:qlqh}
    \begin{aligned}
    & \underset{x\in {\cal{M}}}{\text{min}}
    & &     Q^{\textrm{lqh}}(x,\rho, u) = f(x) + \rho\sum_{i\in{\cal{I}}} {\cal{P}}(g_i(x),u)+\rho\sum_{j\in{\cal{E}}} \sqrt{h_j(x)^2+u^2}.\\ 
    \end{aligned}
    \end{equation}

Note that $Q^{\textrm{lse}}$ and $Q^{\textrm{lqh}}$ are both continuously differentiable and thus we can optimize them with smooth solvers on manifolds. In view of Proposition \ref{l1sum:prop}, in the absence of smoothing, there exists a threshold such that, if $\rho$ is set above that threshold, then optima of the penalized problem coincide with optima of the target problem. However, we often do not know this threshold, and furthermore setting $\rho$ at a high value with a poor initial iterate can lead to poor conditioning and slow convergence. A common practice is to set a relatively low initial $\rho$, optimize, and iteratively increase $\rho$ and re-optimize with a warm start. This is formalized in Algorthm~\ref{algo:smoothing}, which we here call Riemannian Exact Penalty Method via Smoothing (REPMS). In the algorithm and in the discussion below, $\mathrm{update}_\rho$ is a (fixed) Boolean flag indicating whether $\rho$ is dynamically increased or not: this allows to discuss both versions of the algorithm.

\begin{algorithm}[h]
\SetAlgoLined
\SetKw{KwRequire}{Require: }
\SetKw{KwInput}{Input: }
\KwRequire{Riemannian manifold ${\cal{M}}$, twice continuously differentiable functions $f$, $\{g_i\}_{i\in{\cal{I}}}$, $\{h_j\}_{j\in{\cal{E}}} \colon {\cal{M}} \rightarrow \mathbb{R}$.} \\
\KwInput{Starting point $x_0 \in \cal{M}$, accuracy tolerance $\epsilon_{\min}$, starting accuracy $\epsilon_0$, starting penalty coefficient $\rho_0$, starting approx-accuracy $u_0$, minimum approx-accuracy $u_{\min}$, constants $\theta_\epsilon, \theta_u\in (0,1)$, $\theta_\rho > 1$, $\tau \geq 0$, $Q\in\{Q^{\mathrm{lse}}, Q^{\mathrm{lqh}}\}$, minimum step length $d_{\min}$, $\mathrm{update}_\rho\in \{\mathrm{True,False}\}$.}\\
 \For{k = 0,1,2\dots}{
    To obtain $x_{k+1}$, choose any subsolver to approximately solve
    \begin{equation}
    \begin{aligned}
    & \underset{x\in {\cal{M}}}{\text{min}}
    & & Q(x,\rho_k,u_k)\\
    \end{aligned}
    \end{equation}
    with warm-start at $x_k$ and stopping criterion
    \begin{align*}
	    \|\operatorname{grad} Q(x,\rho_k,u_k)\| \leq \epsilon_k.
    \end{align*}\\
    \If{$\mathrm{dist}(x_k, x_{k+1}) < d_{\min}$ and $\epsilon_k \leq \epsilon_{\min}$ and $u_k \leq u_{\min}$}{
        Return $x_{k+1}$\;
    }
    $\epsilon_{k+1} = \max\{\epsilon_{\min}, \theta_\epsilon \epsilon_k\};\ u_{k+1} = \max\{u_{\min}, \theta_u u_k\}$\;
    \eIf{$\mathrm{update}_\rho =\mathrm{True}$ \rm{and} ($k=0$ \text{or} $\max_{j\in {\cal{E}}, i\in {\cal{I}}}\left\{|h_j(x_{k+1})|, g_i(x_{k+1})\right\}\geq \tau$)}{
        $\rho_{k+1} = \theta_\rho\rho_k$
    }{
    $\rho_{k+1} = \rho_k$\;
    }
 }
\caption{Riemannian Exact Penalty Method via Smoothing (REPMS)}\label{algo:smoothing}
\end{algorithm}

The stopping criterion parameters $(d_{\min},u_{\min},\epsilon_{\min})$ are lower-bounds introduced in the algorithm for practical purposes. The algorithm terminates when the step size is too small while the approximation accuracy is high and the tolerance for the subsolver is low. On the other hand, if $u_k$ is too small, numerical difficulties may arise in these two approximation functions. Numerical concerns aside, in theory (exact arithmetic), if we set these parameters to 0, then the following convergence result holds, similar to its Euclidean counterpart of quadratic penalty method \cite[\S17.1]{Nocedal2006NO}. 
\begin{proposition}\label{prop:smoothconv}
In Algorithm \ref{algo:smoothing}, suppose we set $d_{\min} = u_{\min}=\epsilon_{\min} = 0$ and $\mathrm{update}_\rho = \mathrm{False}$. If the sequence $\{x_k\}$ produced by the subsolver admits a feasible limit point $\overline{x}$ where LICQ conditions hold, then $\overline{x}$ satisfies KKT conditions for~\eqref{problem:mcp}.
\end{proposition}
\begin{proof}
See Appendix~\ref{appendlsmoothconv}.
\end{proof}
The conditions of this proposition are only likely to hold if $\rho_0$ (the initial penalty) is sufficiently large, as described in Proposition~\ref{l1sum:prop}. Instead of trying to set $\rho_0$ above the unknown threshold, we may increase $\rho$ in every iteration. When $\mathrm{update}_\rho = \mathrm{True}$, the update of $\rho$ is a heuristic featured in~\cite{pinar1994smoothing} in the Euclidean case: it gives a conditioned increase on the penalty parameter. Intuitively, when the obtained point is far from feasible, the penalty coefficient is likely too small so that the penalty parameter is increased.

\subsection{A subgradient method for sums of maximum functions}\label{sec:REPMSD}
Instead of smoothing, one may attempt to optimize the nonsmooth objective function directly. In this pursuit, the subgradient descent algorithm on Riemannian manifolds, and its variants, have received much attention in recent years; see \cite{absil2017collection, grohs2016varepsilon, hosseini2016line}. These algorithms work for general locally Lipschitz objective functions on manifolds, and would work here as well. However, notice that (\ref{l1sumpenalty:eq}) takes the exploitable form of minimizing a sum of maximum functions:
    \begin{equation}\label{eq:epmobjfunc}
    \begin{aligned}
    & \underset{x\in {\cal{M}}}{\text{min}}
    & & f(x) + \sum_{i\in{\cal{I}}}\text{max}\{0, \rho g_i(x)\} + \sum_{i\in{\cal{E}}}\text{max}\{\rho h_i(x), - \rho h_i(x)\}.
    \end{aligned}
    \end{equation}
We hence propose a robust version of subgradient descent for sums of maximum functions, which may be of interest in its own right. We refer to it as Riemannian Exact Penalty Method via Subgradient Descent (REPMSD).
However, due to the lengthy specification of this method and its poor performance in numerical experiments, we will only give an overview here and report its performance in Section~\ref{sec:XP}. 

In the Euclidean case, subgradient descent with line search performs two steps---finding a good descent direction and performing line search to get the next point. To find a descent direction, one often wants access to the subdifferential at the current point, because it contains a descent direction. This generalizes to the Riemannian case. We refer readers to \cite{ hosseini2016line, grohs2016varepsilon, clarke1990optimization} for the definition of \textit{generalized subdifferential}, which we denote as $\partial f(x)$ for $f\colon{\cal{M}} \to \mathbb{R}$. Note that obtaining full information of the subdifferential is generally difficult, and thus many algorithms \textit{sample} the gradients around the current point to approximate it. However, for a sum of maximum functions, the subdifferential is directly accessible. This class of problems, which includes~\eqref{eq:epmobjfunc}, takes the form:
    \begin{equation}\label{problem:minisummax}
    \begin{aligned}
    & \underset{x\in {\cal{M}}}{\text{min}}
    & & f(x)=\sum_{i=1}^m \max \{f_{i,j}(x)| j = 1, \ldots, n_i\},\\
    \end{aligned}
    \end{equation}
where each function $f_{i,j} \colon \calM \to \reals$ is twice continuously differentiable. Notice that
\begin{equation}\label{eq:minimaxequi}
f(x) = \sum_{i=1}^m \max \{f_{i,j}(x)| j = 1,\ldots, n_i\} = \max_{1\leq j_i\leq n_i\text{ for } i=1\dots m}{\sum_{i=1}^mf_{i,j_i}(x)}.
\end{equation}
For a given $i$ in $1, \ldots, m$ and a given $x\in{\cal{M}}$, let ${\cal{I}}_i$ index the set of functions $f_{i,j}$ which attain the maximum for that $i$:
\begin{align*}
	{\cal{I}}_i & = \{1\leq j \leq n_i \ | \ f_{i,j} (x) = \max\{f_{i,j_i}(x)|j_i = 1\dots n_i\}\}.
\end{align*}
It is clear that for any choice of $j_i$'s (one for each $i$) such that $\sum_i f_{i,j_i}(x) = f(x)$, we have $j_i\in{\cal{I}}_i$ for all $i$. Thus, from the regularity of the maximum function as defined in~\cite{clarke1990optimization} and Proposition~2.3.12 in that same reference, exploiting the fact that pullbacks $f\circ \mathrm{Exp}_x$ are defined on the tangent space $\T_x\calM$ which is a linear subspace, it can be shown using standard definitions of subdifferentials that
\[
    \partial (f\circ \mathrm{Exp}_x) (0_x) = \mathrm{Conv}\left\{\sum_{i=1}^m\operatorname{Grad} (f_{i,j_i}\circ \mathrm{Exp}_x) (0_x) \  \bigg\vert \textrm{ all possible choices of } j_i\in {\cal{I}}_i \right\},
\]
where we use the notation $\mathrm{Grad}$ to stress that this is a Euclidean gradient of the pullback on the tangent space, and the operator $\mathrm{Conv}$ returns the convex hull of given vectors. To be explicit, for each $i$, pick some $j_i$ in $\mathcal{I}_i$: this produces one vector by summation over $i$ as indicated; the convex hull is taken over all possible choices of the $j_i$'s. Then, with Proposition~2.5 of~\cite{hosseini2011generalized}, analogously to~\eqref{eq:gradHesspullbackExp}, we see that the generalized subdifferential of $f$ at $x$ is given by
\begin{align}
    \partial f (x) = \mathrm{Conv}\left\{\sum_{i=1}^m\operatorname{grad} f_{i,j_i}(x) \  \bigg\vert \textrm{ all possible choices of } j_i\in {\cal{I}}_i \right\}.
    \label{eq:subdifffminsummax}
\end{align}
This gives an explicit formula for the subdifferential of $f$ on $\calM$. Whilst the traditional approach is often to get a descent direction directly from this subdifferential or to sample around the point, numerically, it makes sense to relax the condition and to replace ${\cal{I}}_i$ with an approximate subdifferential based on the (larger) set
\[
    {\cal{I}}_{i,\epsilon} = \{1\leq j \leq n_i \ | \  f_{i,j}(x) \geq \max \{f_{i,j}(x)| j = 1, \ldots, n_i\} - \epsilon\},
\]
where $\epsilon > 0$ is a tolerance. Then, one looks for a descent direction in
the subdifferential~\eqref{eq:subdifffminsummax} extended by substituting ${\cal{I}}_{i,\epsilon}$ for ${\cal{I}}_{i}$. For the latter, the classical approach with $\epsilon = 0$ is to look for a tangent vector of minimum norm in the convex hull: this involves solving a convex quadratic program. We here compute a minimum norm vector in the extended convex hull.
For the rest of the algorithm, we perform a classical Wolfe line search and use a limited-memory version of Riemannian BFGS (LRBFGS) updates based on the now well-known observation that BFGS works surprisingly well with nonsmoothness on Euclidean spaces~\cite{lewis2009nonsmooth} as well as on manifolds~\cite{hosseini2016line}---unfortunately we could not run full Riemannian BFGS as this is expensive on high dimensional manifolds. For details on Wolfe line search and BFGS, see Chapters~3, 8 and~9 of \cite{Nocedal2006NO}.

In closing, we note that there exist other methods for nonsmooth optimization, some of which apply on manifolds as well. In particular, we mention proximal point algorithms~\cite{parikh2014proximalalgorithms,bento2017iteration}, Douglas--Rachford-type methods~\cite{bergmann2016parallel} and iteratively reweighted least squares~\cite{chatterjee2013robustrotationaveraging}. We did not experiment with these methods here.

\section{Numerical experiments and discussion}\label{sec:XP}

\subsection{Problems and data}

We describe three applications which can be modeled within the framework of~\eqref{problem:mcp}. For each, we describe how we construct `baskets' of problem instances. These instances will be used to assess the performance of various algorithms. Code and datasets to reproduce the experiments are freely available.\footnote{\url{https://github.com/losangle/Optimization-on-manifolds-with-extra-constraints}.}

\subsubsection{Minimum balanced cut for graph bisection}
The minimum balanced cut problem is a graph problem: given an undirected graph, one seeks to partition the vertices into two clusters (i.e., bisecting the graph) such that the number of edges between clusters is as small as possible, while simultaneously ensuring that the two clusters have about the same size. A spectral method called ratio cut tackles this problem via optimization where the number of crossing edges and the imbalance are penalized in the objective function. However, in~\cite{lang2006fixing}, the author notes that for `power law graphs'---graphs whose degrees follow a power law distribution---this approach fails to enforce balance. Hence, the author proposes a method with stronger constraints to enforce balance: we sketch it below.

Let $L\in \mathbb{R}^{n\times n}$ denote the discrete Laplacian matrix of the given graph ($L = D - A$, where $D$ is the diagonal degree matrix and $A$ is the adjacency matrix). Furthermore, let $x \in \mathbb{R}^n$ be an indicator vector with entries in $\{-1, 1\}$ indicating which cluster each vertex belongs to. One can check that $\frac{1}{4}x^TLx$ evaluates to the number of edges crossing the two clusters. Let $e$ be the vector of 1's with length $n$. To enforce balance, $x^Te = 0$ is imposed, i.e., the number of vertices in the two clusters is the same (this only applies to graphs with an even number of vertices). However, directly solving this integer programming problem is difficult. The authors thus transform this discrete optimization problem into a continuous one. As a first step, the constraints $x_i = \pm 1$ for all $i$ are replaced by $\mathrm{diag}(xx^T) = e$, and the constraint $x^Te = 0$ is replaced by $(x^Te)^2 = e^T(xx^T)e = 0$. Then, $x \in \reals^n$ is relaxed to a matrix $X \in \reals^{n\times k}$ for some chosen $k$ in $1\ldots n$. Carrying over the constraints gives $\mathrm{diag}(XX^T) = e$, so that each row of $X$ is a unit vector, and $e^T(XX^T)e = 0$ is imposed. The problem then becomes:
\begin{equation} \label{eq:minbalcutXXt}
    \begin{aligned}
    & \underset{X\in \mathbb{R}^{n\times k}}{\text{min}}
    & & -\frac{1}{4}\mathrm{tr}(LXX^T)\\
    & \text{subject to}
   & &  \mathrm{diag}(XX^T) = e, \quad e^TXX^Te = 0,
    \end{aligned}
\end{equation}
 For $k = 1$, the problem is equivalent to the original problem. For $k = n$, the problem is equivalent to solving a (convex) semidefinite program (SDP) by considering $XX^T$ as a positive semidefinite matrix ($X$ can be recovered from Cholesky factorization). One may then use a standard clustering algorithm (for example $k$-means) on the rows of $X$ to obtain two clusters. By choosing $k > 1$, the problem becomes continuous (as opposed to discrete), which much extends the class of algorithms that we may apply to solve this problem. Furthermore, since for $k = n$ the problem is essentially convex (being equivalent to an SDP), we heuristically expect that for $k > 1$ the problem may be easier to solve than with $k = 1$.

To keep the dimensionality of the optimization problem low, it is of interest to consider $k > 1$ smaller than $n$, akin to a Burer--Monteiro approach~\cite{burer2003nonlinear}. We notice that $\mathrm{diag}(XX^T) = e$ defines an oblique manifold (a product of $n$ unit spheres in $\reals^k$). Thus, we can view~\eqref{eq:minbalcutXXt} as a problem on the oblique manifold with a single constraint:
\begin{equation} \label{minbalancedcut}
    \begin{aligned}
    & \underset{X\in \mathrm{Oblique}(n, k)}{\text{min}}
    & & -\frac{1}{4}\mathrm{tr}(X^TLX)\\
    & \text{subject to}
   & & e^TXX^Te = 0.
    \end{aligned}
\end{equation}
This fits our problem class~\eqref{problem:mcp} with $\calM = \mathrm{Oblique}(n, k)$ as an embedded Riemannian submanifold of $\reals^{n \times k}$ and a single quadratic equality constraint.\\

\textbf{Input.}
We generate power law graphs with the Barab\'{a}si--Albert algorithm~\cite{barabasialbert2002graphs}. It starts off with a seed of a small connected graph, iteratively adds a vertex to the graph, and attaches $l$ edges to the existing vertices. For each edge, it is attached to a vertex with probability proportional to the current degree of the vertex. We generate the seed with Erd\H{o}s--R\'{e}nyi random graphs $G(n_{\mathrm{seed}},p)$, that is: $G(n_{\mathrm{seed}},p)$ is a random graph with $n_{\mathrm{seed}}$ nodes such that every two nodes are connected with probability $p$ independently of other edges. For each dimension $n\in\{50, 200, 500, 1000, 2000, 5000\}$, we collect a basket of problems by selecting a density $m\in\{0.005, 0.01, 0.02, 0.04, 0.08\}$, such that $l = mn$. For each set of parameters, we repeat the experiment four times with new graphs and new starting points. The seed is generated by $G(2\lceil mn\rceil, 0.5)$.

\subsubsection{Non-negative PCA}
We follow Montanari and Richard~\cite{montanari2016non} for this presentation of non-negative PCA. Let $v_0 \in \reals^d$ be the spiked signal with $\|v_0\| = 1$. In the so-called spiked model, we are given the matrix $Z = \sqrt{SNR}\  v_0^{}v_0^T+N$, where $N$ is a random symmetric noise matrix and $SNR$ is the signal to noise ratio. For $N$, its off-diagonal entries follow a Gaussian distribution ${\cal{N}}(0, \frac{1}{n})$ and its diagonal entries follow a Gaussian distribution ${\cal{N}}(0, \frac{2}{n})$, i.i.d.\ up to symmetry. In classical PCA, we hope to recover the signal $v_0$ from finding the eigenvector that corresponds to the largest eigenvalue. That is, to solve the following problem:
\begin{equation} \label{classicalpca}
    \begin{aligned}
    & \underset{v\in \mathbb{R}^{d}}{\text{min}}
    & & -v^TZv\\
    & \text{subject to}
   & & \|v\| = 1.
    \end{aligned}
\end{equation}
However, it is well known that in the high dimensional regime, $n = O(d)$, this approach breaks down---the principal eigenvector of $Z$ is asymptotically less indicative of $v_0$; see \cite{johnstone2009consistency}. 

One way of dealing with this problem is by introducing structures such as sparsity on the solution space. In~\cite{montanari2016non}, the authors impose nonnegativity of entries as prior knowledge, and propose to solve PCA restricted to the positive orthant:
\begin{equation} \label{nonnegpca}
    \begin{aligned}
    & \underset{v\in \mathbb{R}^{d}}{\text{min}}
    & & -v^TZv\\
    & \text{subject to}
   & & \|v\| = 1, v\geq 0.
    \end{aligned}
\end{equation}
Since $\|v\| = 1$ defines the unit sphere $\mathbf{S}^{d-1}$ in $\reals^d$, one can write the above problem as
\begin{equation} \label{nonnegpcasdpbm}
    \begin{aligned}
    & \underset{v\in \mathbf{S}^{d-1}}{\text{min}}
    & &  -v^TZv\\
    & \text{subject to}
   & & v\geq 0.
    \end{aligned}
\end{equation}
This falls within the framework of~\eqref{problem:mcp}, with $\calM = \mathbf{S}^{d-1}$ as an embedded Riemannian submanifold of $\reals^d$ and $d$ inequality constraints.\\

\textbf{Input.}
Similar to \cite{khuzani2017stochastic}, we synthetically generate data following the symmetric spiked model as described above. We consider dimensions $d = \{10, 50, 200, 500, 1000, 2000\}$, and let $n = d$. For each dimension, we consider the basket of problems parametrized by $SNR\in \{0.05, 0.1, 0.25, 0.5, 1.0, 2.0\}$ which controls the noise level and $\delta \in \{0.1, 0.3, 0.7, 0.9\}$ which controls sparsity. Furthermore, we let the support $S \subseteq \{1, \ldots, n\}$ of the true principal direction $v_0$ be uniformly random, with cardinality $|S| = \lfloor \delta d\rfloor$, and 
\[
v_{0,i} = \begin{cases}\frac{1}{\sqrt{|S|}} &\textrm{ if } i\in S, \\
            0 & \textrm{ otherwise.}
           \end{cases}
\]
Note that $S$ controls how sparse $v_0$ is. For each set of parameters, the experiment was repeated four times with different random noise $N$. As there are 6 choices for $d$, 4 for $\delta$, 6 for $SNR$, this gives 144 problems in each basket.

\subsubsection{$k$-means via low-rank SDP}
$k$-means is a traditional clustering algorithm in machine learning. Given a set of data points $\{x_i\}_{i=1}^{n}$, $x_i\in \mathbb{R}^l$ for some $l$, and a desired number of clusters $k$, this algorithm aims to partition the data into $k$ subsets $\{S_j\}_{j=1}^k$ such that the following is minimized:
\[
    {\cal{D}}(\{S_j\}_{j=1}^k) = \sum_{j=1}^k\sum_{x_i\in S_j}\|x_i-m_j\|^2,
\]
where $m_j = \frac{1}{|S_j|}\sum_{x_i\in S_j}x_i$ is the center of mass of the points selected by $S_j$. Following~\cite{carson2017manifold}, an equivalent formulation is as follows ($I$ is the identity matrix):
\begin{equation} \label{kmeans}
    \begin{aligned}
    & \underset{Y\in\mathbb{R}^{n\times k}}{\text{min}}
    & & -\mathrm{tr}(Y^TDY)\\
    & \text{subject to}
   & &  Y^TY = I,\quad Y\geq 0, \quad YY^Te = e,
    \end{aligned}
\end{equation}
where $D_{ij} = \|x_i-x_j\|^2$ is the squared Euclidean distance matrix. The constraint $Y^TY = I$ indicates $Y$ has orthonormal columns. The set of such matrices is called the Stiefel manifold. Thus, we can rewrite the problem as:
\begin{equation} \label{kmeansstiefel}
    \begin{aligned}
    & \underset{Y\in \mathrm{Stiefel}(n,k)}{\text{min}}
    & & -\mathrm{tr}(Y^TDY)\\
    & \text{subject to}
   & &  Y\geq 0, \quad YY^Te = e.
    \end{aligned}
\end{equation}
This falls within the framework of~\eqref{problem:mcp} with $\calM$ the Stiefel manifold (as an embedded Riemannian submanifold of $\reals^{n\times k}$), $n$ equality constraints and $nk$ inequality constraints.\\

\textbf{Input.} We take data sets from the UCI Machine Learning repository \cite{Lichman:2013}. For each dataset, we clean the data by removing categorical data and normalizing each feature. This gives $X = [x_1\dots x_N]$, and we get $D$ via $D_{ij} = \|x_i-x_j\|^2$. The specifications of the cleaned data are described in Table~\ref{table:kmeanspec}.

\begin{table}
	\centering
{\begin{tabular}{c c c c}
\hline
\hline                        
Dataset name & Number of data  & features & clusters \\
 [0.5ex]
\hline
Iris & 150 & 4 & 3  \\
ecoli & 336 & 5 & 8  \\
pima & 768 & 8 & 2  \\
sonar & 208 & 60 & 2 \\
vehicle &846 & 18 & 4 \\
wine & 178 & 13 & 3 \\
cloud & 2048 & 10 & 4 \\
 [1ex]
\hline
\end{tabular}}

\caption{Details of the datasets used in $k$-means}
\label{table:kmeanspec}
\end{table}

\subsection{Methodology} \label{sec:methodology}
We compare the methods discussed above against each other. In addition, as a control to test the hypothesis that exploiting the geometry of the manifold constraint is beneficial, we also compare against Matlab's built-in constrained optimization solver \texttt{fmincon}, which treats the manifold as supplementary equality constraints. We choose \texttt{fmincon} because it is a general purpose solver which combines various algorithms to enhance performance and it has been refined over years of developments, and thus acts as a good benchmark. To summarize, the methods are:
\begin{itemize}
	\item RALM: Riemannian Augmented Lagrangian Method, Section~\ref{sec:ALM};
	\item REPMS($Q^{\textrm{lqh}}$): exact penalty method with smoothing (linear-quadratic and pseudo-Huber), Section~\ref{sec:REPMS};
	\item REPMS($Q^{\textrm{lse}}$): exact penalty method with smoothing (log-sum-exp), Section~\ref{sec:REPMS}; 
	\item REPMSD: exact penalty method via nonsmooth optimization, Section~\ref{sec:REPMSD};
	\item Matlab's \texttt{fmincon}: does not exploit manifold structure.
\end{itemize}

We supply \texttt{fmincon} with the gradients of both the objective function and the constraints. We use the default settings: minimum step size is $10^{-10}$ and relative constraint violation tolerance is $10^{-6}$.\footnote{When the step size is of order $10^{-10}$, we believe that the current point is close to convergence. We also conducted experiments with minimum step size $10^{-7}$ for minimum balanced cut and non-negative PCA, and the performance profiles are visually similar to those displayed here.} Furthermore, we disable the stopping criterion based on a maximum number of iterations or queries in order to allow \texttt{fmincon} to converge to good solutions. For all of our methods, we define $d_{\mathrm{min}} = 10^{-10}, \epsilon_0 = 10^{-3}, \epsilon_{\min} = 10^{-6}, \theta_\epsilon = (\epsilon_{\min}/\epsilon_0)^\frac{1}{30}, \rho_0 = 1, \theta_{\rho} = 3.3$. Specifically, for RALM, $\theta_{\sigma} = 0.8$; for exact penalty methods, $\tau = 10^{-6}$, $u_0 = 10^{-1}, u_{\min} = 10^{-6}, \theta_u= (u_{\min}/u_0)^\frac{1}{30}$. As subsolver for RALM and the smoothing methods REPMS, we use LRBFGS: a Riemannian, limited-memory BFGS~\cite{huang2015broyden} as implemented in Manopt~\cite{boumal2014manopt}. For REPMSD, we use a minimum-norm tangent vector in the extended subgradient, then we use a type of LRBFGS inverse Hessian approximation for Hessian updates to choose the update direction~\cite{hosseini2016line}. For these limited-memory subsolvers, we let the memory be 30, the maximum number of iterations be 200, and the minimum step size be $10^{-10}$. For each experiment, all solvers have the same starting point randomly chosen on the manifold. An experiment is also terminated if it runs over one hour. We register the last produced iterate. Maximum constraint violation (Maxvio), cost function value and computation time are recorded for each solver, where Maxvio is defined as:
\begin{align*}
	\mathrm{Maxvio = } \max\left(\{|h_j(x)|\big|j\in {\cal{E}}\}\cup\{g_i(x)|i\in {\cal{I}}\}\cup\{0\}\right).
\end{align*}

For minimum balanced cut and non-negative PCA, we present performance profiles to compare computation time across problem instances for the various solvers. Performance profiles were popularized in optimization by Dolan and Mor\'e~\cite{dolan2002benchmarking}. Quoting from that reference, performance profiles show ``(cumulative) distribution functions for a performance metric as a tool for benchmarking and comparing optimization software.'' In particular, we show ``the ratio of the computation time of each solver versus the best time of all of the solvers as the performance metric.'' The profiles of REPMSD are only shown in lower dimensions to reduce computation time. See Figure~\ref{bmincut} for further details.

We consider that a point is feasible if Maxvio is smaller than $5\cdot 10^{-4}$. By denoting $f_{\min}$ as the minimum cost among all feasible solutions returned by solvers, we consider a solution to be best-in-class if it is feasible $and$ the cost $f$ satisfies $\left|f/f_{\min}-1\right| < 2\%$. In view of the objective functions of these two problems, a relative error tolerance is reasonable. We consider that a solver solves the problem if the returned point is best-in-class.

For $k$-means, we give a table of experimental results for Maxvio, cost and time across all datasets.

\subsection{Results and analysis}
\subsubsection{Minimum balanced cut for graph bisection}
We display in Figure~\ref{bmincut} the performance profiles for various dimensions, corresponding to the number of vertices in the graph. We note that apart from REPMSD, all methods perform well in the low dimensional case. In the high dimensional case, \texttt{fmincon} fails: in all test cases it converges to points far from feasible despite having no iteration limit; exact penalty methods and RALM are still able to give satisfactory results in some cases. 

\begin{figure}[h!]
	\centering
	\includegraphics[width=1.0\linewidth]{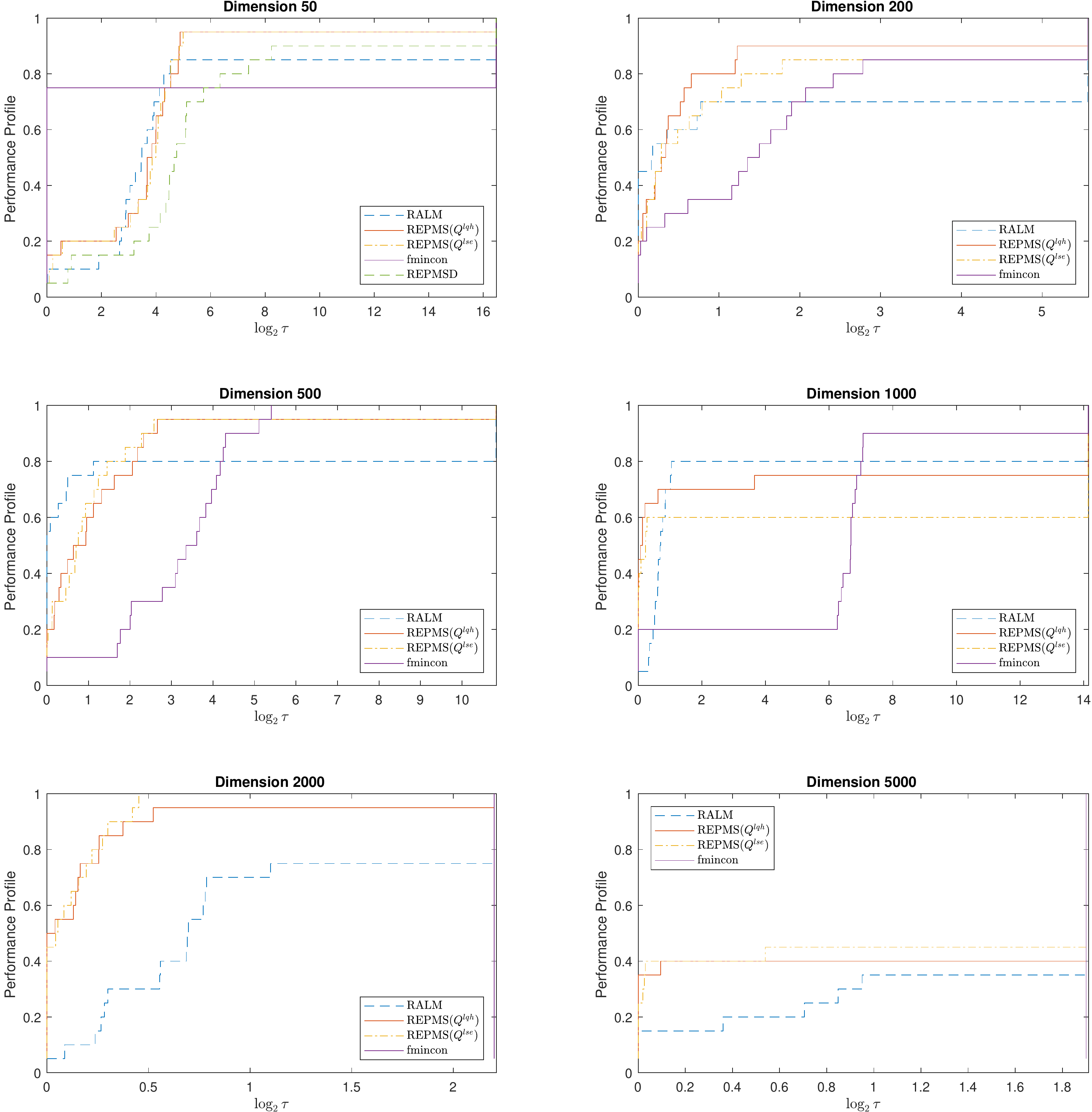}
	\caption{Performance profiles of computation time for different dimensions of the minimum balanced cut problem (with maximum constraint violation and cost accounted for---see Section~\ref{sec:methodology}). For each dimension, all solvers are run on a basket of problems. For each problem that could be solved, one solver obtained the fastest computation time. The curve for a solver passes through point $(\log_2\tau, q)$ if it solved a fraction $q$ of the problems within a factor $\tau$ of the fastest solver (which may change for each problem). Thus, upper-left is best. In particular, for $\tau = 1$ ($\log_2\tau = 0$), the curve of solver $A$ passes at level $q$ if solver $A$ was the fastest on a fraction $q$ of the problems. These may not sum to 100\% since some problems were not solved by any solver (and there could be ties, though that is unlikely).}
	\label{bmincut}
\end{figure}


\subsubsection{Non-negative PCA}
Table~\ref{nnpca} displays the performance profiles for non-negative PCA for various dimensions, corresponding to the length of the sought vector. RALM and REPMS($Q^{\textrm{lqh}}$) outperform \texttt{fmincon} as dimension increases. REPMS($Q^{\textrm{lse}}$) is rather slow. Notice that in this problem there are lots of inequality constraints involved. For REPMS($Q^{\textrm{lqh}}$), when we compute the gradient of the objective function, gradients of constraints that are not violated at the current point are not computed--- they are 0. However, they are computed in REPMS($Q^{\textrm{lse}}$), which slows down the algorithm. This observation suggests that, for problems with a large number of inequality constraints, REPMS($Q^{\textrm{lse}}$) may not be the best choice.

\begin{figure}[h!]
	\centering
	\includegraphics[width=1\linewidth]{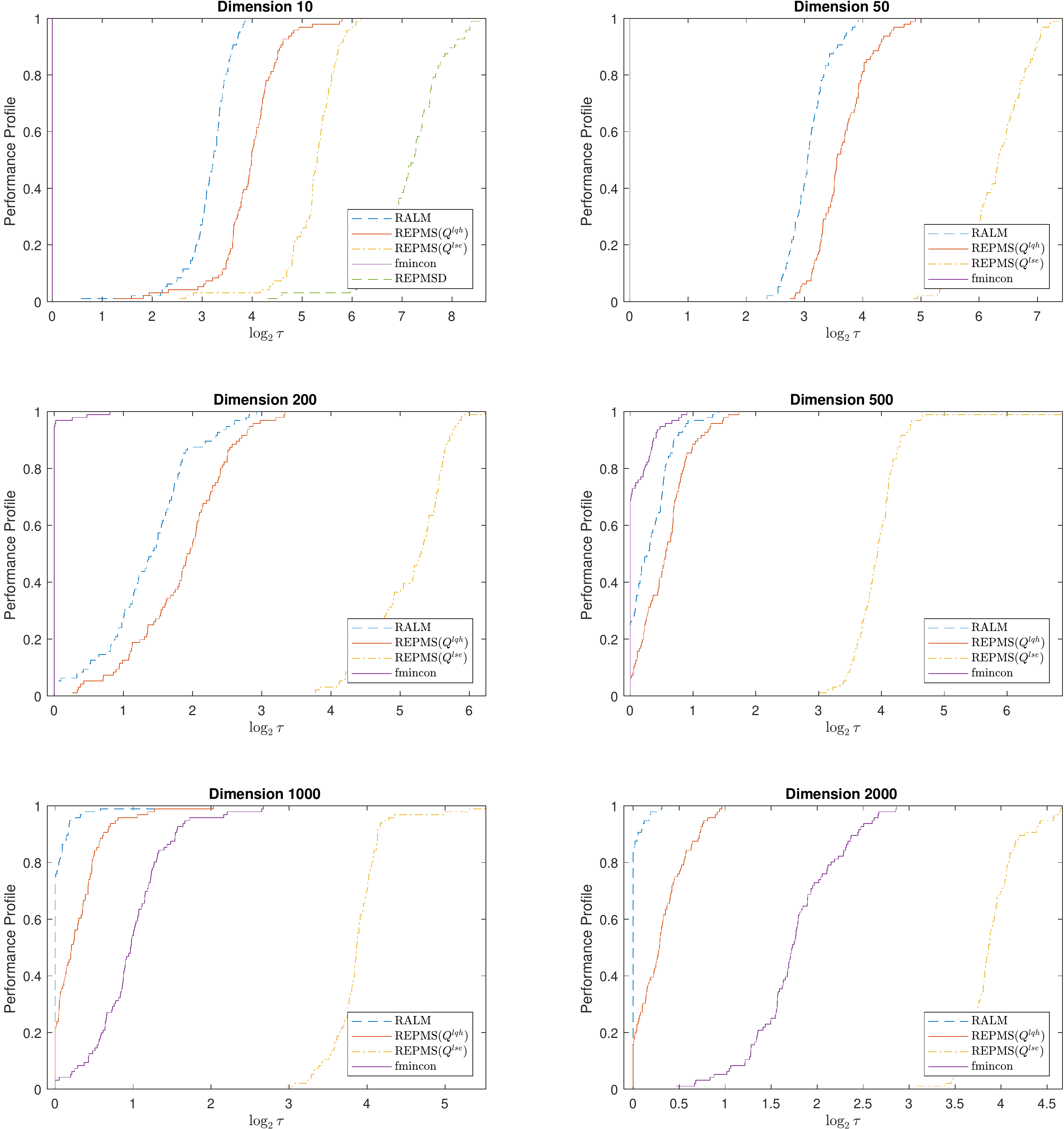}
	 \caption{Performance profile of computational time for different dimensions of non-negative PCA (with maximum constraint violation and cost accounted for---see Section~\ref{sec:methodology}). See caption of Figure~\ref{bmincut} for details on performance profiles.}
	 \label{nnpca}
\end{figure}


\subsubsection{$k$-means}
Table~\ref{kmeansresult} gives the results for $k$-means clustering. As the dimension gets larger, \texttt{fmincon} performs worse than RALM and REPMS in both Maxvio and Cost. Among all these methods, REPMS($Q^{\textrm{lqh}}$) and RALM are the better methods since they have relatively shorter solving time and lower Maxvio. In addition, REPMS($Q^{\textrm{lqh}}$) achieves the lowest cost for high dimensions.

\begin{table}
	\centering
	{\begin{tabular}{lllllll}
			&          & RALM                                      & REPMS$(Q^{\textrm{lqh}})$                         & REPMS$(Q^{\textrm{lse}})$                         & fmincon                                  & REPMSD                                   \\ \hline
			Iris    & Maxvio   & $1.92 \cdot 10^{-5}$                                 & $6.49 \cdot 10^{-6}$                                 & $4.26 \cdot 10^{-2}$                                 & $4.44 \cdot 10^{-5}$                                 & {$\mathbf{2.13 \cdot 10^{-7}}$} \\
			& Cost     & $3.31 \cdot 10^{-1}$                                 & {$\mathbf{4.43 \cdot 10^{-2}}$} & $8.20 \cdot 10^{+0}$                                 & $1.14 \cdot 10^{+1}$                                 & $2.80 \cdot 10^{-1}$                                 \\
			& Time (s) & $4.72 \cdot 10^{+2}$                                 & $4.89 \cdot 10^{+2}$                                 & {$\mathbf{1.70 \cdot 10^{+2}}$} & $3.13 \cdot 10^{+2}$                                 & $3.86 \cdot 10^{+3}$                                 \\ \hline
			Wine    & Maxvio   & $7.25 \cdot 10^{-5}$                                 & $6.62 \cdot 10^{-4}$                                 & $4.81 \cdot 10^{-2}$                                 & {$\mathbf{1.75 \cdot 10^{-5}}$}                               & -                                        \\
			& Cost     & $1.89 \cdot 10^{+0}$                                 & {$\mathbf{1.68 \cdot 10^{-1}}$}                    & $8.43 \cdot 10^{+0}$                                 & $1.19 \cdot 10^{+1}$                                 & -                                        \\
			& Time (s) & $3.45 \cdot 10^{+2}$                                 & $2.87 \cdot 10^{+2}$                                 & $1.85 \cdot 10^{+2}$                                 & {$\mathbf{1.59 \cdot 10^{+2}}$} & -                                        \\ \hline
			Sonar   & Maxvio   & $3.02 \cdot 10^{-5}$                                 & $1.42 \cdot 10^{-5}$                                 & $5.48 \cdot 10^{-2}$                                 & {$\mathbf{6.65 \cdot 10^{-7}}$} & -                                        \\
			& Cost     & {$\mathbf{8.99 \cdot 10^{-1}}$} & $9.50 \cdot 10^{+0}$                                 & $1.50 \cdot 10^{+1}$                                 & $1.80 \cdot 10^{+1}$                                 & -                                        \\
			& Time (s) & $5.28 \cdot 10^{+2}$                                 & $3.48 \cdot 10^{+2}$                                 & {$\mathbf{2.55 \cdot 10^{+2}}$} & $6.70 \cdot 10^{+2}$                                 & -                                        \\ \hline
			Ecoli   & Maxvio   & {$\mathbf{8.33 \cdot 10^{-4}}$}                     & $2.65 \cdot 10^{-3}$                                 & $6.73 \cdot 10^{-2}$                                 & $9.67 \cdot 10^{-1}$                                 & -                                        \\
			& Cost     & $4.08 \cdot 10^{+1}$                                 &{$\mathbf{1.24 \cdot 10^{-1}}$}                       & $1.87 \cdot 10^{+1}$                                 & $6.45 \cdot 10^{+0}$                                 & -                                        \\
			& Time (s) & $1.75 \cdot 10^{+3}$                                 & $1.44 \cdot 10^{+3}$                                 & {$\mathbf{1.01 \cdot 10^{+3}}$}                    & $3.62 \cdot 10^{+3}$                                 & -                                        \\ \hline
			Pima    & Maxvio   & {$\mathbf{2.10 \cdot 10^{-10}}$} & $1.82 \cdot 10^{-3}$                                 & $9.10 \cdot 10^{-2}$                                 & $2.34 \cdot 10^{-5}$                                 & -                                        \\
			& Cost     & $2.60 \cdot 10^{-2}$                                 & {$\mathbf{9.48 \cdot 10^{-3}}$} & $1.56 \cdot 10^{+0}$                                 & $2.13 \cdot 10^{+1}$                                 & -                                        \\
			& Time (s) & $1.37 \cdot 10^{+3}$                                 & $1.25 \cdot 10^{+3}$                                 & {$\mathbf{5.79 \cdot 10^{+2}}$} & $3.62 \cdot 10^{+3}$                                 & -                                        \\ \hline
			Vehicle & Maxvio   & $8.87 \cdot 10^{-3}$                                 & {$\mathbf{4.35 \cdot 10^{-4}}$} & $9.17 \cdot 10^{-2}$                                 & $2.71 \cdot 10^{+2}$                                 & -                                        \\
			& Cost     & $3.18 \cdot 10^{+0}$                                 & {$\mathbf{3.77 \cdot 10^{-2}}$} & $2.23 \cdot 10^{+1}$                                 & $6.71 \cdot 10^{+2}$                                 & -                                        \\
			& Time (s) & $2.88 \cdot 10^{+3}$                                 & $2.70 \cdot 10^{+3}$                                 & {$\mathbf{2.61 \cdot 10^{+3}}$} & $3.63 \cdot 10^{+3}$                                 & -                                        \\ \hline
			Cloud   & Maxvio   & $1.81 \cdot 10^{-2}$                                 & {$\mathbf{1.83 \cdot 10^{-3}}$} & $6.79 \cdot 10^{-1}$                                 & $3.70 \cdot 10^{+3}$                                 & -                                        \\
			& Cost     & $1.36 \cdot 10^{-1}$                                 & {$\mathbf{5.49 \cdot 10^{-4}}$} & $1.97 \cdot 10^{+0}$                                 & $2.20 \cdot 10^{+3}$                                 & -                                        \\
			& Time (s) & $4.08 \cdot 10^{+3}$                                 & {$\mathbf{3.65 \cdot 10^{+3}}$} & $3.78 \cdot 10^{+3}$                                 & $3.84 \cdot 10^{+3}$                                 & -                                        \\ \hline
	\end{tabular}}
	\caption{Maximum constraint violation, cost function value, and computation time results of different algorithms for $k$-means. For dataset `cloud' (which does not come with a specified number of clusters), we set the number of clusters $k$ to 4. For all other datasets, we used the number of clusters recommended by the dataset documentation.}
	\label{kmeansresult}
\end{table}

\section{Conclusion}
In this paper, augmented Lagrangian methods and exact penalty methods are extended to Riemannian manifolds, along with some essential convergence results. These are compared on three applications against a baseline, which is to treat the manifold as a set of general constraints. The three applications are minimum balanced cut, non-negative PCA and $k$-means. The conclusions from the numerical experiments are that, in high dimension, it seems to be beneficial to exploit the manifold structure; as to which Riemannian algorithm is better, the results vary depending on the application.

 We consider this work to be a first systematic step towards understanding the empirical behavior of straightforward extensions of various techniques of constrained optimization from the Euclidean to the Riemannian case. One direct advantage of exploiting the Riemannian structure is that this part of the constraints is satisfied up to numerical accuracy at each iteration. While our numerical experiments indicate that some of these methods perform satisfactorily as compared to a classical algorithm which does not exploit Riemannian structure (especially in high dimension), the gains are moderate. For future research, it is of interest to pursue refined versions of some of the algorithms we studied here, possibly inspired by existing refinements in the Euclidean case but also by a more in depth study of the geometry of the problem class. Furthermore, it is interesting to pursue the study of the convergence properties of these algorithms---and the effects of the underlying Riemannian geometry---beyond the essential results covered here, in particular because the Riemannian approach extends to general abstract manifolds, including some which may not be efficiently embedded in Euclidean space through practical equality constraints.

\section*{Acknowledgments} 
We thank an anonymous reviewer for detailed and helpful comments on the first version of this paper. NB is partially supported by NSF grant DMS-1719558.

\bibliographystyle{abbrv}
\bibliography{biblio}

\appendix

\section{Proof of Proposition \ref{prop:almfirst}} \label{appendixalmfirst}

We first introduce two supporting lemmas. The first lemma is a well-known fact for which we provide a proof for completeness.\footnote{The proof follows an argument laid out by John M.\ Lee: \url{https://math.stackexchange.com/questions/2307289/parallel-transport-along-radial-geodesics-yields-a-smooth-vector-field}.}
\begin{lemma} \label{lem:smothparalleltransport}
	Let $p$ be a point on a Riemannian manifold ${\cal{M}}$, and let $v$ be a tangent vector at $p$. Let ${\cal{U}}$ be a normal neighborhood of $p$, that is, the exponential map maps a neighbourhood of the origin of $\T_p{\cal{M}}$ diffeomorphically to ${\cal{U}}$. Define the following vector field on ${\cal{U}}$:
	\begin{align*}
		\forall q \in {\cal{U}}, \qquad V(q) & = {\cal{P}}_{p \to q} v,
	\end{align*}
	where parallel transport is done along the (unique) minimizing geodesic from $p$ to $q$. Then, $V$ is a smooth vector field on ${\cal{U}}$.
\end{lemma}
\begin{proof}
	Parallel transport from $p$ is along geodesics passing through $p$. To facilitate their study, set up normal coordinates $\phi\colon U \subset \mathbb{R}^d \rightarrow{\cal{U}}$ around $p$ (in particular, $\phi(0) = p$), where $d$ is the dimension of the manifold. For a point $\phi(x_1,\dots,x_d)$, by definition of normal coordinates, the radial geodesic from $p$ is $c(t) = \phi(tx_1,\dots, tx_d)$. Our vector field of interest is defined by $V(p) = v$ and the fact that it is parallel along every radial geodesic $c$ as described.
	
	For a choice of point $\phi(x)$ and corresponding geodesic, let
	\begin{align*}
	V(c(t)) = \sum_{k=1}^d v_k(t) \partial_k(c(t))
	\end{align*}
	for some coordinate functions $v_1, \ldots, v_d$, where $\partial_k$ is the $k$th coordinate vector field. These coordinate functions satisfy the following ordinary differential equations (ODE)~\cite[Prop.~2.6, eq.~(2)]{carmo1992riemannian}:
	\begin{align*}
	0 & = \frac{dv_k(t)}{dt} + \sum_{i,j}\Gamma^k_{ij}(tx_1,\dots, tx_d) v_j(t) x_i, & k = 1,\dots, d,
	\end{align*}
	where $\Gamma$ denotes Christoffel symbols.
	Expand $V(p)$ into the coordinate vector fields: $v = \sum_{k=1}^d w_k \partial_k(p)$. Then, the initial conditions are $v_k(0) = w_k$ for each $k$. Because these ODEs are smooth, solutions $v_k(t; w)$ exist, and they are smooth in both $t$ and the initial conditions $w$~\cite[Thm.~D.6]{lee2012smoothmanifolds}. But this is not enough for our purpose.
	
	Crucially, we wish to show smoothness also in the choice of $x \in U$. To this end, following a classical trick, we extend the set of equations to let $x$ be part of the variables, as follows:
	\[
	\begin{cases}
	0 = \frac{dv_k(t)}{dt} + \sum_{i,j}\Gamma^k_{ij}(tu_1(t),\ldots, tu_d(t)) v_j(t) u_i(t), & k=1,\ldots, d,\\
	0 = \frac{du_k(t)}{dt}, & k=1, \ldots, d.
	\end{cases}
	\]
	The extended initial conditions are:
	\begin{align*}
	v_k(0) & = w_k, & u_k(0) & = x_k, & k & = 1, \ldots, d.
	\end{align*}
	Clearly, the functions $u_k(t)$ are constant: $u_k(t) = x_k$.
	These ODEs are still smooth, hence solutions $v_k(t; w, x)$ still exist and are identical to those of the previous set of ODEs, except we now see they are also smooth in the choice of $x$. Specifically, for every $x \in U$,
	\begin{align*}
	V(\phi(x)) & = V(c(1)) = \sum_{k=1}^{d} v_k(1; w, x) \partial_k(\phi(x)),
	\end{align*}
	and each $v_k(1; w, x)$ depends smoothly on $x$. Hence, $V$ is smooth on ${\cal{U}} = \phi(U)$.
	%
	%
\end{proof}
\begin{lemma} \label{lemma:ode}
	Given a Riemannian manifold ${\cal{M}}$, a function $f \colon \cal{M} \to \mathbb{R}$ (continuously differentiable), and a point $p\in{\cal{M}}$, if $p_0, p_1, p_2, \ldots$ is a sequence of points in a normal neighborhood of $p$ and convergent to $p$, then the following holds:
	\[
    	\lim_{k\rightarrow\infty} \left\|{\cal{P}}_{p_k\rightarrow p}\operatorname{grad} f(p_k) - \operatorname{grad} f(p)\right\|_p = 0,
	\]
	where ${\cal{P}}_{p_k\rightarrow p}$ is the parallel transport from $\T_{p_k}\calM$ to $\T_p\calM$ along the minimizing geodesic.
\end{lemma}
\begin{proof}[Proof of Lemma~\ref{lemma:ode}]
    As parallel transport is an isometry, it is equivalent to show
    \begin{equation}\label{eqleminv}
    \lim_{k\rightarrow\infty} \left\|\operatorname{grad} f(p_k) -  {\cal{P}}_{p\rightarrow p_k} \operatorname{grad} f(p)\right\|_{p_k} = 0.
    \end{equation}
    Under our assumptions, $\grad f$ is a continuous vector field. 
    Furthermore, by Lemma~\ref{lem:smothparalleltransport}, in a normal neighborhood of $p$, the vector field
    $V(y) = {\cal{P}}_{p\rightarrow y} \operatorname{grad} f(p)$ is a continuous vector field as well. Hence, $\grad f - V$ is a continuous vector field around $p$; since $\grad f(p) - V(p) = 0$, the result is proved: $\lim_{k \to \infty} \grad f(p_k) - V(p_k) = \grad f(p) - V(p) = 0$.
\end{proof}

\begin{proof}[Proof of Proposition \ref{prop:almfirst}]
Restrict to a convergent subsequence if needed, so that $\lim_{k\rightarrow \infty}x_k = \overline{x}$. Further exclude a (finite) number of $x_k$'s so that all the remaining points are in a neighborhood of $\overline{x}$ where the exponential map is a diffeomorphism. In this proof, let ${\cal{A}}$ denote ${\cal{A}}(\overline{x})$ for ease of notation: this is the set of active constraints at the limit point. Then, there exist constants $c, k_1$ such that $g_i(x_k) < c < 0$ for all $k>k_1$ with $i \in {\cal{I}} \setminus {\cal{A}}$.

When $\{\rho_k\}$ is unbounded, since multipliers are bounded, there exists $k_2>k_1$ such that $\lambda_i^k+\rho_kg_i(x_{k+1}) < 0$ for all $k\geq k_2$, $i\in{\cal{I}}\setminus {\cal{A}}$. Thus, by definition, $\lambda_i^{k+1} = 0$ for all $k\geq k_2$, $i\in{\cal{I}}\setminus{\cal{A}}$.

When instead $\{\rho_k\}$ is bounded, $\lim_{k\rightarrow \infty} |\sigma_i^k| = 0$. Thus for $i\in {\cal{I}}\setminus {\cal{A}}$, in view of $g_i(x_k) < c<0$ for all $k>k_1$, we have $\lim_{k\rightarrow\infty} \frac{-\lambda_i^k}{\rho_k}= 0$. Then, for large enough $k$, $\lambda_i^k +\rho_kg_i(x_{k+1})<0$ and thus there exists $k_2>k_1$ such that $\lambda_i^k = 0$ for all $k\geq k_2$. So in either case, we can find such $k_2$.

As LICQ is satisfied at $\overline{x}$, by continuity of the gradients of $\{g_i\}$ and $\{h_j\}$, the tangent vectors $\{\operatorname{grad} h_j(x_k)\}_{j\in{\cal{E}}}\cup\{\operatorname{grad} g_i(x_k)\}_{i\in{\cal{I}}\cap{\cal{A}}}$ are linearly independent for all $k>k_3>k_2$ for some $k_3$. Define
\begin{align*}
	\overline{\lambda}_i^k & = \max\left\{0, \lambda_i^{k-1} + \rho_{k-1} g_i(x_{k})\right\}, & \textrm{ and } & & \overline{\gamma}_j^k & = \gamma_j^{k-1}+\rho_{k-1} h_j(x_{k})
\end{align*}
as the unclipped update. Define $S_k := \max\{\|\overline{\gamma}^k\|_\infty,\|\overline{\lambda}^k\|_\infty\}$. We are going to discuss separately for situations when $S_k$ is bounded and when it is unbounded. If it is bounded, then denote a limit point of $\overline{\lambda}^k, \overline{\gamma}^k$ as $\overline{\lambda}$ and $\overline{\gamma}$. Let
\[
    v = \operatorname{grad} f(\overline{x}) + \sum_{j\in{\cal{E}}} \overline{\gamma}_j\operatorname{grad} h_j(\overline{x}) + \sum_{i\in{\cal{I}}\cap {\cal{A}}}\overline{\lambda}_i \operatorname{grad} g_i(\overline{x}).
\]
In order to prove that $v$ is zero, we compare it to a similar vector defined at $x_k$, for all large $k$, and consider the limit $k \to \infty$. Unlike the Euclidean case in the proof in~\cite{andreani2007augmented}, we cannot directly compare tangent vectors in the tangent spaces at $x_k$ and $\overline{x}$: we use parallel transport to bring all tangent vectors to the tangent space at $\overline{x}$: 
\begin{eqnarray*}
\|v\| & \leq &\left\|\operatorname{grad} f(\overline{x}) - {\cal{P}}_{x_k\rightarrow \overline{x}} \operatorname{grad} f(x_k) + \sum_{j\in{\cal{E}}} \overline{\gamma}_j\left(\operatorname{grad} h_j(\overline{x})-{\cal{P}}_{x_k\rightarrow \overline{x}} \operatorname{grad} h_j(x_k)\right) \right.\\
&& + \left.\sum_{i\in{\cal{I}}\cap {\cal{A}}}\overline{\lambda}_i \left(\operatorname{grad} g_i(\overline{x}) - {\cal{P}}_{x_k\rightarrow \overline{x}} \operatorname{grad} g_i(x_k)\right)\right\|\\
&& + \left\|{\cal{P}}_{x_k\rightarrow \overline{x}} \operatorname{grad} f(x_k) + \sum_{j\in{\cal{E}}} \overline{\gamma}_j{\cal{P}}_{x_k\rightarrow \overline{x}} \operatorname{grad} h_j(x_k) + \sum_{i\in{\cal{I}}\cap {\cal{A}}}\overline{\lambda}_i  {\cal{P}}_{x_k\rightarrow \overline{x}} \operatorname{grad} g_i(x_k)\right\|.
\end{eqnarray*}
By Lemma~\ref{lemma:ode}, the first term vanishes in the limit $k \to \infty$ since $x_k \to \overline{x}$. We can understand the second term using isometry of parallel transport and linearity:
\begin{eqnarray*}
&&\left\|{\cal{P}}_{x_k\rightarrow \overline{x}} \operatorname{grad} f(x_k) + \sum_{j\in{\cal{E}}} \overline{\gamma}_j{\cal{P}}_{x_k\rightarrow \overline{x}} \operatorname{grad} h_j(x_k) + \sum_{i\in{\cal{I}}\cap {\cal{A}}}\overline{\lambda}_i  {\cal{P}}_{x_k\rightarrow \overline{x}} \operatorname{grad} g_i(x_k)\right\|_{\overline{x}}\\
&=& \left\|\operatorname{grad} f(x_k) + \sum_{j\in{\cal{E}}} \overline{\gamma}_j\operatorname{grad} h_j(x_k) + \sum_{i\in{\cal{I}}\cap {\cal{A}}}\overline{\lambda}_i \operatorname{grad} g_i(x_k)\right\|_{x_k}\\
&\leq & \left\|\sum_{j\in{\cal{E}}} (\overline{\gamma}_j - \overline{\gamma}^k_j)\operatorname{grad} h_j(x_k) + \sum_{i\in{\cal{I}}\cap {\cal{A}}}(\overline{\lambda}_i-\overline{\lambda}_i^k) \operatorname{grad} g_i(x_k)\right\|_{x_k} \\
&&+ \left\|\operatorname{grad} f(x_k) + \sum_{j\in{\cal{E}}} \overline{\gamma}_j^k\operatorname{grad} h_j(x_k) + \sum_{i\in{\cal{I}}}\overline{\lambda}_i^k \operatorname{grad} g_i(x_k)\right\|_{x_k} + \left\|\sum_{i\in{\cal{I}}\setminus{\cal{A}}} \overline{\lambda}_i^k \operatorname{grad} g_i(x_k)\right\|_{x_k}.
\end{eqnarray*}
Here, the second term vanishes in the limit because it is upper bounded by $\epsilon_{k}$ (by assumption) and we let $\lim_{k\rightarrow\infty} \epsilon_k = 0$; the last term vanishes in the limit because of the discussion in the second paragraph; and the first term attains arbitrarily small values for large $k$ as norms of gradients are bounded in a neighbourhood of $\overline{x}$ and by definition of $\overline{\lambda}$ and $\overline{\gamma}$. Since $v$ is independent of $k$, we conclude that $\|v\| = 0$. Therefore, $\overline{x}$ satisfies KKT conditions.

On the other hand, if $\{S_k\}$ is unbounded, then for $k\geq k_3$, we have
\[
    \left\|\frac{1}{S_k}\operatorname{grad} f(\overline{x}) +  \sum_{j\in{\cal{E}}} \frac{\overline{\gamma}_j}{S_k}\operatorname{grad} h_j(\overline{x}) +  \sum_{i\in{\cal{I}}}\frac{\overline{\lambda}_i}{S_k} \operatorname{grad} g_i(\overline{x})\right\| \leq \frac{\epsilon_k}{S_k}.
\]
As all the coefficients on the left-hand side are bounded in $[-1,1]$, and by definition of $S_k$, the coefficient vector has a nonzero limit point. Denote it as $\overline{\lambda}$ and $\overline{\gamma}$. By a similar argument as above, taking the limit in $k$, we can obtain
\[
    \left\|\sum_{j\in{\cal{E}}} \overline{\gamma}\operatorname{grad} h_j(\overline{x}) + \sum_{i\in{\cal{I}}\cap {\cal{A}}}\overline{\lambda} \operatorname{grad} g_i(\overline{x})\right\|  = 0,
\]
which contradicts the LICQ condition at $\overline{x}$. Hence, the situation that $\{S_k\}$ is unbounded does not take place, so we are left with the cases where it is bounded, for which we already showed that $\overline{x}$ satisfies KKT condition.
\end{proof}

\section{Proof of Proposition \ref{prop:almsecond}} \label{appendixalmsecond}
\begin{proof}
The proof is adapted from Section 3 in~\cite{andreani2017second}.  Define $\overline{\gamma}_j^k = \gamma_j^{k-1}+\rho_{k-1} h_j(x_k)$. By Proposition \ref{prop:almfirst}, $\overline{x}$ is a KKT point and by taking a subsequence of $\{x_k\}$ if needed, $\overline{\gamma}^k$ is bounded and converges to $\overline{\gamma}$.

For any tangent vector $d\in {\cal{C}}^W(\overline{x})$, we have $\langle d, \operatorname{grad} h_j(\overline{x})\rangle = 0$ for all $j\in {\cal{E}}$. Let $m = |{\cal{E}}|$, and dimension of ${\cal{M}}$ be $n\geq m$.
Let $\varphi$ be a chart such that $\varphi(\overline{x}) = 0$. From~\cite[Prop.~8.1]{lee2012smoothmanifolds}, the component functions of $h_j$ with respect to this chart are smooth. Let $\partial_1\dots \partial_{n}$ be the basis vectors of the given local chart. Let $d = (d_1\partial_1,\dots, d_n\partial_n)$. Define:
    ${\cal{F}}\colon\mathbb{R}^{n+m} \rightarrow \mathbb{R}^m$, i.e. for $x\in\mathbb{R}^n, y\in \mathbb{R}^m, j\in\{1,\dots, m\}$ as 
    \[
        {\cal{F}}_j(x,y) = \langle (y_1\partial_1, \dots, y_m\partial_m, d_{m+1}\partial_{m+1}, \dots, d_n\partial_n), \operatorname{grad} h_j(\varphi^{-1}(x))\rangle_{\varphi^{-1}(x)}.
    \]
If we denote $h_{j}^l$ as the $l$-th coordinate of vector $\operatorname{grad} h_j$ in this system, and $G_x$ as gram matrix for the metric where $G_{x_{p,q}} = \langle\partial_p,\partial_q\rangle_x$, then the above expression can be written as
\[
    {\cal{F}}_j(x,y) = [y_1, \dots, y_m, d_{m+1}, \dots, d_n] G_{\varphi^{-1}(x)} [h_j^1(\varphi^{-1}(x)), \dots h_j^n(\varphi^{-1}(x))]^T.
\]
and by abuse of notation where $[1\dots m]$ means extracting the first $m$ columns, we have\[\frac{\partial {\cal{F}}_j}{\partial y} = \left([h_j^1(\varphi^{-1}(x)), \dots h_j^n(\varphi^{-1}(x))]G_{\varphi^{-1}(x)}\right)_{[1\cdots m]}.\] Notice that $[h^1(\varphi^{-1}(\overline{x})), \dots h^n(\varphi^{-1}(\overline{x}))]$ has full row rank (by LICQ), so it has rank $m$. As $G_{\varphi^{-1}(\overline{x})}$ is invertible, $\frac{\partial {\cal{F}}}{\partial y}(\overline{x})$ must be invertible (reindex the columns from the top of the proof for this $m\times n$ matrix if needed so that the $m$ columns form a full rank matrix). Then, by the implicit function theorem, for a small neighbourhood $U$ of $\varphi^{-1}(\overline{x})$, we have a continuously differentiable function $g:U\rightarrow \mathbb{R}^m$, where $g(\varphi^{-1}(\overline{x})) = [d_1, \dots, d_m]$ and
\[
    {\cal{F}}(x, g(\varphi^{-1}(x))) = 0.
\]
For each $x$ locally around $\overline{x}$, let 
\[
d_x = [g(\varphi^{-1}(x))_1\partial_1, \dots, g(\varphi^{-1}(x))_m\partial_m, d_{m+1}\partial_{m+1}, \dots, d_n\partial_n]\in \T_{x}{\cal{M}}.
\]
These vectors then forms a smooth vector field such that $\langle d_x, \operatorname{grad} h_j(x)\rangle = 0$ for all $j\in{\cal{E}}$, and $d = d_{\overline{x}}$. Then we have that
\begin{eqnarray*}
&&\mathrm{Hess}_x {\cal{L}}_{\rho_{k-1}}(x_k,\gamma^{k-1})(d_{x_k},d_{x_k}) \\
&=& \langle d_{x_k}, \mathrm{Hess} f(x_k) d_{x_k}\rangle + \rho_{k-1} \sum_{j\in {\cal{E}}} \langle d_{x_k}, \nabla_{d_{x}}(h_j(x)+\frac{\gamma_j^{k-1}}{\rho_{k-1}})\operatorname{grad} h_j(x)\rangle_{x_k} \\
&=& \langle d_{x_k}, \mathrm{Hess} f(x_k) d_{x_k}\rangle + \rho_{k-1}\sum_{j\in {\cal{E}}} d_{x}[h_j(x)+\frac{\gamma_j^{k-1}}{\rho_{k-1}}](x_k)\langle d_{x}, \operatorname{grad} h_j(x)\rangle_{x_k} \\
&&\quad\quad+ \sum_{j\in{\cal{E}}} (\rho_{k-1} h_j(x_k)+\gamma_j^{k-1})\langle d_{x}, \nabla_{d_x}\operatorname{grad} h_j(x)\rangle_{x_k}\\
&=& \langle d_{x_k}, \mathrm{Hess} f(x_k) d_{x_k}\rangle +  \sum_{j\in{\cal{E}}} (\rho_{k-1} h_j(x_k)+\gamma_j^{k-1})\langle d_{x}, \nabla_{d_{x}}\operatorname{grad} h_j(x)\rangle_{x_k}\\
&=& \langle d_{x_k}, \nabla_{d_{x_k}} \operatorname{grad} f(x_k)\rangle + \sum_{j\in {\cal{E}}} \overline{\gamma}_j^{k}\langle d_{x}, \nabla_{d_{x}}\operatorname{grad} h_j(x)\rangle_{x_k}
\end{eqnarray*}
where the second equality is by definition of connection; the third is by orthogonality of $d$ with $\{\operatorname{grad} h_j\}$; the fourth is from the definition of Hessian and $\overline{\gamma}$. Therefore we have
\[
    \langle d_{x}, \nabla_{d_{x}} \operatorname{grad} f(x)\rangle_{x_k} + \sum_{j\in{\cal{E}}} \overline{\gamma}_j^k\langle d_{x}, \nabla_{d_{x}}\operatorname{grad} h_j(x)\rangle_{x_k} \geq -\epsilon_k \|d_{x_k}\|^2
\]
Since the connection maps two continuously differentiable vector fields to a continuous vector field, we can take a limit and state:
\[
    \langle d, \nabla_{d} \operatorname{grad} f(\overline{x})\rangle + \sum_{j\in{\cal{E}}} \overline{\gamma}_j\langle d, \nabla_{d}\operatorname{grad} h_j(\overline{x})\rangle \geq 0
\]
which is just $\mathrm{Hess} {\cal{L}}(\overline{x}, \overline{\gamma})(d,d) \geq 0$.
\end{proof}

\section{Proof of Proposition~\ref{l1sum:prop}}\label{appendl1sum:prop}
In the proof below, we use the following notation:
\begin{equation} \label{set:f1}
  v \in F'(x^*,\lambda^*, \gamma^*) \Leftrightarrow \begin{cases}
            v\in \mathrm{T}_{x^*}{\cal{M}}, &\\
    \langle \operatorname{grad} h_j(x^*), v\rangle = 0& \textrm{ for all }j\in {\cal{E}},\textrm{ and} \\
    \langle \operatorname{grad} g_i(x^*), v\rangle \leq 0& \textrm{ for all }i\in {\cal{A}}(x^*)\cap{\cal{I}}.\\
  \end{cases}
\end{equation}

\begin{proof}
    Consider the function $Q$, defined by:
    \[
        Q(x, \rho) = f(x) + \rho \left(\sum_{i\in{\cal{I}}}\text{max}\{0, g_i(x)\} + \sum_{j\in{\cal{E}}}|h_j(x)|\right).        
    \]
    In a small enough neighbourhood of $x^*$, terms for inactive constraints disappear and $Q$ is just:
    \[
        Q(x, \rho) = f(x) + \rho \left(\sum_{i\in{\cal{A}}(x^*)\cap{\cal{I}}}\text{max}\{0, g_i(x)\} + \sum_{j\in{\cal{E}}}|h_j(x)|\right).
    \]
    Although $Q$ is nonsmooth, it is easy to verify that it has directional derivative in all directions: 
    \begin{eqnarray*}
    \lim_{\tau\rightarrow 0 }\frac{\text{max}\{0,g_i(\mathrm{Exp}_{x^*}(\tau d))\}-\text{max}\{0,g_i(x^*)\}}{\tau} & = &  \lim_{\tau\rightarrow 0 }\frac{\text{max}\{0,g_i(\mathrm{Exp}_{x^*}(\tau d))\}}{\tau}
    \end{eqnarray*}
    and since $g_i\circ\mathrm{Exp}_{x^*}$ is sufficiently smooth, discussing separately the sign of $\frac{d}{d\tau}(g_i\circ \mathrm{Exp}_{x^*})(\tau d)$, we have the right hand side equal to $\text{max}\{0, \frac{d}{d\tau}(g_i\circ \mathrm{Exp}_{x^*})(\tau d)\} = \max\{0,\langle\operatorname{grad} g_i(x^*), d\rangle\}$. Similarly, we have 
    \[
        \lim_{\tau\rightarrow 0 }\frac{|h_j(\mathrm{Exp}_{x^*}(\tau d))|-|h_j(x^*)|}{\tau}  =  \left|\frac{d}{d\tau}(h_j\circ \mathrm{Exp}_{x^*})(\tau d)\right| = |\langle\operatorname{grad} h_j(x^*), d\rangle|.
    \]
    Hence, the directional derivative along direction $d$, $Q(x^*,\rho; d)$, is well defined:
    \begin{equation}\label{exactProof:1}
    Q(x^*,\rho;d) =\langle\operatorname{grad} f(x^*), d\rangle + \rho\left(\sum_{i\in{\cal{A}}(x^*)\cap{\cal{I}}}\max\{0,\langle\operatorname{grad} g_i(x^*), d\rangle\} + \sum_{j\in{\cal{E}}}|\langle\operatorname{grad} h_j(x^*), d\rangle|\right).
    \end{equation}
    As $x^*$ is a KKT point,
    \[
    \operatorname{grad} f(x^*) + \sum_{i\in{\cal{A}}(x^*)\cap{\cal{I}}} \lambda_i^* \operatorname{grad} g_i(x^*) + \sum_{j\in{\cal{E}}} \gamma_j^* \operatorname{grad} h_j(x^*) = 0.
    \]
    Thus,
    \begin{eqnarray*}
    0 & = & \langle \operatorname{grad} f(x^*),d\rangle + \sum_{i\in{\cal{A}}(x^*)\cap{\cal{I}}} \lambda_i^* \langle\operatorname{grad} g_i(x^*),d\rangle + \sum_{j\in{\cal{E}}} \gamma_j^* \langle\operatorname{grad} h_j(x^*),d\rangle \\
    & \leq & \langle \operatorname{grad} f(x^*),d\rangle + \sum_{i\in{\cal{A}}(x^*)\cap{\cal{I}}} \lambda_i^* \max\{0, \langle\operatorname{grad} g_i(x^*),d\rangle\}+ \sum_{j\in{\cal{E}}} \gamma_j^* |\langle\operatorname{grad} h_j (x^*),d\rangle|.
    \end{eqnarray*}
    Combining with equation (\ref{exactProof:1}), we have
    \begin{equation}\label{exactProof:2}
        Q(x^*, \rho;d) \geq \sum_{i\in{\cal{A}}(x^*)\cap{\cal{I}}} (\rho-\lambda_i^*) \max\{0, \langle\operatorname{grad} g_i(x^*),d\rangle\} + \sum_{j\in{\cal{E}}} (\rho-\gamma_j^*) |\langle\operatorname{grad} h_j(x^*),d\rangle|.
    \end{equation}
    For contradiction, suppose $x^*$ is not a local minimum of $Q$. Then, there exists $\{y_k\}_{k=1}^\infty$, $\lim_{k\rightarrow \infty} y_k = x^*$ such that $Q(y_k, \rho) < Q(x^*, \rho) = f(x^*)$. By restricting to a small enough neighbourhood, there exists $\eta_k = \mathrm{Exp}^{-1}_{x^*}(y_k)$. Considering only a subsequence if needed, we have $\lim_{k\rightarrow \infty} \frac{\eta_k}{\|\eta_k\|} = \bar{\eta}$. It is easy to see that $Q(\mathrm{Exp}_{x^*}(\cdot),\rho)$ is locally Lipschitz continuous at $0_{x^*}$, which gives
    \[
        Q(\mathrm{Exp}_{x^*}(\|\eta_k\| \bar{\eta}), \rho) = Q(\mathrm{Exp}_{x^*}(\eta_k), \rho) + o(\|\eta_k\|) = Q(y_k, \rho) + o(\|\eta_k\|).
    \]
    Subtract $Q(x^*,\rho)$ and take the limit:
    \[
        \lim_{k\rightarrow \infty} \frac{Q(\mathrm{Exp}_{x^*}(\|\eta_k\| \bar{\eta}), \rho) - Q(x^*,\rho)}{\|\eta_k\|} =\lim_{k\rightarrow \infty} \frac{Q(y_k, \rho) - Q(x^*,\rho)}{\|\eta_k\|}  + \lim_{k\rightarrow \infty} \frac{o(\|\eta_k\|)}{\|\eta_k\|}\leq 0.
    \]
    Notice the left-most expression is just $Q(x^*, \rho; \bar{\eta})$. Since coefficients on the right-hand side of~\eqref{exactProof:2} are strictly positive, we must have $\langle \operatorname{grad} g_i(x^*), \bar{\eta}\rangle \leq 0$ and $\langle \operatorname{grad} h_j(x^*), \bar{\eta}\rangle = 0$. Since the exponential mapping is of second order, we have a Taylor expansion for $f$,
    \[
        f(y_k) = f(x^*) + \langle \operatorname{grad} f(x^*), \eta_k\rangle + \frac{1}{2}\langle \eta_k, \mathrm{Hess} f(x^*)[\eta_k]\rangle + o(\|\eta_k\|^2),
    \]
    and similarly for $g_i$ and $h_j$. Notice that 
    \begin{eqnarray*}
    Q(y_k, \rho) & = & f(y_k) + \rho \left(\sum_{i\in{\cal{A}}(x^*)\cap{\cal{I}}}\max\{0, g_i(y_k)\} + \sum_{j\in{\cal{E}}}|h_j(y_k)|\right)\\
    & \geq & \left(f(y_k) + \sum_{i\in{\cal{A}}(x^*)\cap{\cal{I}}}\lambda_i^* g_i(y_k) + \sum_{j\in{\cal{E}}}\gamma_j^*h_j(y_k)\right) +   \sum_{i\in{\cal{A}}(x^*)\cap{\cal{I}}}(\rho - \lambda_i^*)\max\{0, g_i(y_k)\} \\
    & & + \sum_{j\in{\cal{E}}}(\rho  - \gamma_j^*)|h_j(y_k)|\\
    & \geq & f(x^*) + \langle \operatorname{grad} f(x^*) + \sum_{i\in{\cal{A}}(x^*)\cap{\cal{I}}} \lambda_i^* \operatorname{grad} g_i(x^*) + \sum_{j\in{\cal{E}}} \gamma_j^* \operatorname{grad} h_j(x^*), \eta_k\rangle \\
    & & + \frac{1}{2}\langle \eta_k, \mathrm{Hess} (f(x^*)+\sum_{i\in{\cal{A}}(x^*)\cap{\cal{I}}} \lambda_i^* g_i(x^*)+\sum_{j\in{\cal{E}}}\gamma_{j}^* h_j(x^*))[\eta_k]\rangle + o(\|\eta_k\|^2)  + P(y_k)\\
    & = & f(x^*) + 0 + \frac{1}{2}\langle \eta_k, \mathrm{Hess}({\cal{L}}(x,\lambda^*,\gamma^*)(x ^*)[\eta_k])\rangle + o(\|\eta_k\|^2) + P(y_k)
    \end{eqnarray*}
        where $P(y_k) = \sum_{i\in{\cal{A}}(x^*)\cap{\cal{I}}}(\rho - \lambda_i^*)\max\{0, g_i(y_k)\} + \sum_{j\in{\cal{E}}}(\rho - \gamma_j^*)|h_j(y_k)|$. The first inequality follows from quadratic approximation of $f + \sum_{i\in{\cal{A}}(x^*)\cap{\cal{I}}} \lambda_{i} g_i+\sum_{j\in{\cal{E}}} \gamma_j h_j$ and bilinearity of the metric. The last equality comes from the definition of KKT points. Dividing the equation through by $\|\eta\|^2$, we obtain
    \begin{equation}\label{exactproof:3}
    \lim_{k\rightarrow\infty} \frac{Q(y_k,\rho) - f(x^*)}{\|\eta\|^2}  = \lim_{k\rightarrow\infty}\frac{1}{2}\left\langle \frac{\eta_k}{\|\eta_k\|}, \mathrm{Hess}({\cal{L}}(x,\lambda^*,\gamma^*)(x ^*)\left[\frac{\eta_k}{\|\eta_k\|}\right]\right\rangle + 0 + \lim_{k\rightarrow\infty}\frac{P(y_k)}{\|\eta\|^2}.
    \end{equation}
    
    If $\bar{\eta}\in F'$, then as $P(y_k)\geq 0$, the first term on the right hand side will be strictly larger than 0, which is a contradiction to $Q(y_k,\rho) < f(x^*)$ for all $k$. If $\bar{\eta}\in F-F'$, then there exists $g_{i'}$ such that $\langle \operatorname{grad} g_{i'}(x), \bar{\eta}\rangle > 0$. Then,
    \[
        g_{i'}(y_k) = g_{i'}(x^*) + \langle \operatorname{grad} g_{i'}(x^*), \eta_k\rangle + o(\|\bar{\eta}\|) = \langle \operatorname{grad} g_{i'}(x^*), \eta_k\rangle + o(\|\eta_k\|).
    \]
    Hence, dividing the above expression by $\|\eta_k\|$ gives
    \[
        \lim_{k\rightarrow\infty} \frac{g_{i'}(y_k)}{\|\eta_k\|} \geq \lim_{k\rightarrow\infty} \langle \operatorname{grad} g_{i'}(x^*), \frac{\eta_k}{\|\eta_k\|}\rangle + 0 = \langle \operatorname{grad} g_{i'}(x^*), \bar{\eta_k}\rangle > 0.
    \]
    Notice that $\frac{P(y_k)}{\|\eta\|^2} \geq \frac{g_{i'}(y_k)}{\|\eta_k\|}$ for large enough $k$ and a contradiction is obtained by plugging it into (\ref{exactproof:3}).
\end{proof}

\section{Proof of Proposition~\ref{prop:smoothconv}}\label{appendlsmoothconv}
\begin{proof}
We give a proof for $Q^{\mathrm{lse}}$---it is analogous for $Q^{\mathrm{lqh}}$.
For each iteration $k$ and for each $i\in{\cal{I}}$ and $j\in{\cal{E}}$, define the following coefficients:
\begin{align*}
\lambda_i^k & = \frac{e^{g_i(x_{k+1})/u_k}}{1+e^{g_i(x_{k+1})/u_k}}, &  \textrm{ and } & & \gamma^k_j & = \frac{e^{h_j(x_{k+1})/u_k}-e^{-h_j(x_{k+1})/u_k}}{e^{h_j(x_{k+1})/u_k}+e^{-h_j(x_{k+1})/u_k}}.
\end{align*}    
Then, a simple calculation shows that (under our assumptions, $\rho_k = \rho_0$ for all $k$; we simply write $\rho$):
\begin{align*}
    \operatorname{grad} Q^{\textrm{lse}}(x_{k+1},\rho_k,u_k) & = \operatorname{grad} f(x_{k+1}) + \rho\sum_{i\in{\cal{I}}} \lambda_i^k \operatorname{grad} g_i(x_{k+1}) + \rho\sum_{j\in{\cal{E}}} \gamma^k_j \operatorname{grad} h_j(x_{k+1}).
\end{align*}
Notice that the multipliers are bounded: $\gamma^k_j \in [-1,1]$ and $\lambda_i^k\in[0,1]$. Hence, as sequences indexed by $k$, they have a limit point: we denote them by $\overline{\gamma}\in [-1,1]$ and $\overline{\lambda}\in [0,1]$. Furthermore, since $\overline{x}$ is feasible, there exists $k_1$ such that for any $k>k_1$, $i\in{\cal{I}}\setminus {\cal{A}}(\overline{x})$, $g_i(x_k) < c$ for some constant $c<0$. Then, as $u_k\rightarrow 0$, by definition, $\lambda^k_i$ goes to 0 for $i\in {\cal{I}}\setminus {\cal{A}}(\overline{x})$. This shows $\overline{\lambda}_i = 0$ for $i\in {\cal{I}}\setminus {\cal{A}}(\overline{x})$. Considering a convergent subsequence if needed, there exists $k_2>k_1$ such that, for all $k>k_2$, $\textrm{dist}(x_k,\overline{x}) < {\mathit{i}}(\overline{x})$ (the injectivity radius). Thus, parallel transport from each $x_k$ to $\overline{x}$ is well defined. Consider
\[
    v = \operatorname{grad} f(\overline{x}) + \rho \sum_{i\in{\cal{I}}\cap {\cal{A}}(\overline{x})}\overline{\lambda}_i \operatorname{grad} g_i(\overline{x}) + \rho \sum_{j\in{\cal{E}}} \overline{\gamma}_j\operatorname{grad} h_i(\overline{x}).
\] Notice that its coefficients are bounded, so we can get $\|v\| = 0$ similar to the proof of Proposition~\ref{prop:almfirst}.
\end{proof}

\end{document}